\documentclass[a4paper,11pt]{article}

\usepackage{indentfirst}
\usepackage{amsmath}
\usepackage{amsthm}
\usepackage{amsfonts}
\usepackage{bbm}
\usepackage{accents}
\usepackage{xcolor}

\newtheorem{thm}{Theorem}
\newtheorem{remark}{Remark}
\newtheorem{lemma}{Lemma}
\newtheorem{cor}{Corollary}
\newtheorem{proposition}{Proposition}

\theoremstyle{definition}

\newcommand{\eee}{{\rm e}}

\newcommand{\me}{\mathbb{E}}
\newcommand{\mn}{\mathbb{N}}
\newcommand{\mr}{\mathbb{R}}
\newcommand{\mmp}{\mathbb{P}}
\def\1{\mathbbm{1}}
\newcommand{\ms}{\mathfrak{s}}

\begin{document}

\title{Laws of the iterated logarithm for random Dirichlet series with general weights}\date{}
\author{Alexander Iksanov\footnote{Faculty of Computer Science and Cybernetics, Taras Shevchenko National University of Kyiv, Ukraine; e-mail address:
iksan@univ.kiev.ua} \ \ and \ \ Ruslan Kostohryz\footnote{Faculty of Computer Science and Cybernetics, Taras Shevchenko National University of Kyiv, Ukraine; e-mail address:
kostogriz2909@gmail.com}}
\maketitle
\begin{abstract}
\noindent For each $s>0$, we consider a random Dirichlet series $X(s)=\sum_{k\geq 1}k^{-1/2-s}a_k\eta_k$, where $\eta_1$, $\eta_2,\ldots$ are independent and identically distributed random variables with mean zero and finite positive variance, and $(a_k)_{k\geq 1}$ is a deterministic sequence of real numbers satisfying $\sum_{k\geq 1}k^{-1-2s}a_k^2<\infty$ for each $s>0$ and $\sum_{k\geq 1}k^{-1}a_k^2=\infty$. We investigate the almost-sure fluctuations of $X(s)$ as $s\to0+$. Under these minimal assumptions, we construct examples exhibiting several non-standard forms of the law of the iterated logarithm (LIL) along
suitable sequences: the normalization and the upper and lower limit constants may differ from their classical counterparts. We also show that a regular growth condition of the form $\sum_{k\leq n}k^{-1}a_k^2\sim c(\log n)^\beta$, where $c,\beta>0$, is not by itself sufficient to ensure a standard LIL. Finally, under an additional counting condition controlling the frequency of indices at which the weights $a_k$ are comparatively large, we prove that $(2{\rm Var}\,[X(s)]\log\log({\rm Var}\,[X(s)]))^{-1/2}X(s)$ has the almost-sure cluster set $[-1,1]$ as $s\to 0+$. The latter result is applied to several coefficient sequences of number-theoretic origin.
\end{abstract}


\hphantom{2020 Mathematics Subject Classification: } Secondary: 11A25; 11M41; 60G50

\section{Introduction and main results}

Let $\eta_1,\eta_2,\ldots$ be independent copies of a random variable $\eta$ satisfying $\me[\eta]=0$ and $\sigma^2:=\me [\eta^2] \in(0,\infty)$. Let $(a_k)_{k\ge1}$ be a sequence of {\it real} numbers such that

\begin{equation}\label{eq:standing}
\sum_{k\geq 1}\frac{a_k^2}{k}=\infty\quad\text{and}\quad \sum_{k\geq 1}\frac{a_k^2}{k^{1+2s}}<\infty\quad\text{for}~s>0.
\end{equation}

For $s>0$, define the random Dirichlet series $$X(s):=\sum_{k\geq 1} \frac{a_k\eta_k}{k^{1/2+s}}.$$ The second condition in \eqref{eq:standing} guarantees  that the series converges almost surely for each $s>0$ and that $v(s):={\rm Var}[X(s)]<\infty$. 

Random Dirichlet series corresponding to specific choices of $(a_k)$ have been investigated extensively in recent years. A brief overview of the available limit theorems is given in Table~\ref{table:surv}. Throughout the paper, LIL stands for the law of the iterated logarithm.

\begin{table}[ht]
\centering

\begin{tabular}{|p{0.14\textwidth}|p{0.26\textwidth}|p{0.41\textwidth}|c|}
\hline
\textbf{$a_k$}
&
\textbf{Distribution of $\eta$}
&
\textbf{Results}
&
\textbf{Source}
\\
\hline

$a_k=1$
&
$\mmp\{\eta=\pm 1\}=1/2$
&
LIL
&
\cite{Aymone+Frometa+Misturini:2020}
\\\hline 

$a_k=(\log k)^\alpha$ for $\alpha>-1/2$
&
$\me [\eta]=0,~{\rm Var}[\eta]\in (0,\infty)$
&
Functional central limit theorem and LIL
&
\cite{Buraczewski etal:2023}
\\
\hline

$a_k=(\log k)^\alpha$ for $\alpha>-1/2$
&
$\me [\eta]=0,~{\rm Var}[\eta]\in (0,\infty)$
&
Functional large deviation principle (under $\me [\eee^{\lambda |\eta|}]<\infty$ for all $\lambda>0$); functional moderate deviation principle (under $\me [\eee^{\lambda |\eta|}]<\infty$ for some $\lambda>0$) and functional LIL
&
\cite{Gao+Xia:2026}
\\
\hline

$a_k=(\log k)^\alpha$ for $\alpha\geq -1/2$
&
$\me [\eta]=0,~{\rm Var}[\eta]\in (0,\infty)$ and $\me [\eee^{\lambda \eta}]<\infty$ for $|\lambda|<\lambda_0$
&
Precise large deviations
&
\cite{Xia:2026}
\\
\hline

$a_k=(\log k)^{-1/2}$
&
$\me [\eta]=0,~{\rm Var}[\eta]\in (0,\infty)$
&
Functional central limit theorem and LIL
&
\cite{Iksanov+Kostohryz:2025}
\\
\hline
$a_k=1$\newline for prime $k$ and $=0$ otherwise
&
$\me [\eta]=0,~{\rm Var}[\eta]\in (0,\infty)$
&
Functional central limit theorem and LIL
&

\cite{Dong+Iksanov:2026+}
\\
\hline

$a_k=1$
&
$\eta$ has symmetric $\gamma$-stable distribution for $\gamma\in (0,2)$
&
LIL-like result for $\sum_{k\geq 1}k^{-1/\gamma-s} a_k\eta_k$
&
\cite{Zhao+Huang:2024}
\\
\hline

$a_k=1$
&
$\me [\eta]=0,~{\rm Var}[\eta]\in (0,\infty)$ and $\eta$ is essentially upper bounded by a positive deterministic constant
&
Precise asymptotic behavior of $\mmp\{X(s)>t\}$ as $t\to\infty$ for fixed $s\in (0,1/2]$
&
\cite{Iksanov+Wachtel:2024}
\\
\hline

$a_k=1$
&
$\me [\eta]=0,~\mmp\{-\eta>t\}\sim {\rm const}\,t^{-\beta}$ as $t\to\infty$ for $\beta\in (1,2)$, and $\eta$ is essentially upper bounded by a positive deterministic constant
&
Precise asymptotic behavior of $\mmp\{X(s)>t\}$ as $t\to\infty$ for fixed $s\in (1/\beta-1/2,1/2]$
&
\cite{Kostohryz:2026+}
\\
\hline

\end{tabular}
\caption{Survey of some limit theorems for random Dirichlet series.}
\label{table:surv}
\end{table}

A result that serves as the starting point of the present paper is Theorem 3.1 in \cite{Buraczewski etal:2023}, reproduced below. For a family $(x_s)$ of real numbers denote by $C((x_s))$ the set of its limit points as $s\to 0+$.
\begin{proposition}\label{prop:bura}
Let $\alpha>-1/2$. Assume that $\me[\eta]=0$ and $\sigma^2=\me[\eta^2]\in(0,\infty)$. Then
\begin{equation}\label{eq:limitpoints}
C\Big(\Big(\Big(\frac{2^{2\alpha}}{\sigma^2 \Gamma(1+2\alpha)}\frac{s^{1+2\alpha}}{\log\log1/s}\Big)^{1/2}\sum_{k\geq 2} \frac{(\log k)^\alpha \eta_k}{k^{1/2+s}}: s\in(0, \eee^{-1})\Big)\Big)=[-1,1] \quad\text{\rm a.s.},
\end{equation}
where $\Gamma$ is the Euler gamma function.
In particular,
\begin{equation}\label{eq:limsup}
\limsup_{s\to 0+}\Big(\frac{s^{1+2\alpha}}{\log\log 1/s}\Big)^{1/2}\sum_{k\geq 2} \frac{(\log k)^\alpha \eta_k}{k^{1/2+s}}=\frac{(\Gamma(1+2\alpha))^{1/2}}{2^\alpha}\sigma
\quad\text{\rm a.s.}
\end{equation}
\end{proposition}

Observe that
\begin{equation}\label{eq:basic2}
\sum_{k\leq n}\frac{(\log k)^{2\alpha}}{k}~\sim~\frac{1}{1+2\alpha}(\log n)^{1+2\alpha},\quad n\to\infty\quad\text{and}\quad \sum_{k\geq 2}\frac{(\log k)^{2\alpha}}{k^{1+2s}}~\sim~ \frac{\Gamma(1+2\alpha)}{(2s)^{1+2\alpha}},\quad s\to 0+.
\end{equation}
The second asymptotic relation shows that
$v(s)~\sim~\sigma^2 \Gamma(1+2\alpha) (2s)^{-(1+2\alpha)}$ as $s\to 0+$,
so that \eqref{eq:limsup} is precisely a standard LIL, with the normalization
$(2v(s)\log\log v(s))^{1/2}$ and the limit constant $1$. 
Motivated by \eqref{eq:limsup}, we replace the specific choice $a_k=(\log k)^\alpha$ by an arbitrary sequence $(a_k)$ satisfying \eqref{eq:standing}. Our objectives are twofold. First, we investigate the range of possible LILs that may arise under these minimal assumptions \eqref{eq:standing}. Second, we identify additional conditions on $(a_k)$ under which the standard LIL continues to hold.

The first condition in \eqref{eq:standing} is the weakest natural assumption under which a LIL can be expected. By itself, however, it imposes very little regularity on the sequence $(a_k)$. Our first results, Theorems \ref{thm:main1} and \ref{thm:main2}, demonstrate  that, under \eqref{eq:standing} alone, the asymptotic behavior of $X(s)$ may differ drastically from that predicted by the standard LIL, and a variety of non-standard LILs is possible. More precisely, we shall show that under \eqref{eq:standing} the following situations are possible:

\noindent 1) the normalization for $\limsup$ or $\liminf$ is as in \eqref{eq:limitpoints}, that is, $(2v(s)\log\log v(s))^{1/2}$, yet the limit constant is different (Theorem \ref{thm:main1}(a));

\noindent 2) the normalization for $\limsup$ or $\liminf$ is different from $(2v(s)\log\log v(s))^{1/2}$ (Theorem \ref{thm:main1}(b));

\noindent 3) the normalizations for $\limsup$ or $\liminf$ are different (Theorem \ref{thm:main1}(c));

\noindent 4) $X(s)$ does not satisfy a LIL at all, whereas $\log |X(s)|$ does (Theorem \ref{thm:main2}).

The standard form of LIL in \eqref{eq:limitpoints} is partially due to the fact that if $a_k=(\log k)^\alpha$ for $\alpha>-1/2$, the variance of each summand in $X(s)$ is negligibly small compared to the variance of $X(s)$, namely,
\begin{equation*}
\lim_{s\to 0+}\frac{\sup_k\frac{(\log k)^{2\alpha}}{k^{1+2s}}}{v(s)}=0.
\end{equation*}
In order to obtain the effects mentioned in 1)-4) we shall define a sequence $(a_k)$ satisfying \eqref{eq:standing} and such that the asymptotic behavior of $v$ is driven by a single term or a few terms.

\begin{thm}\label{thm:main1}
a) Assume that $\eta$ has a normal distribution with mean $0$ and variance $\sigma^2\in (0,\infty)$ and set $s_n=2^{-1}\eee^{-n^2}$ for $n\in\mn$. Then there exists a sequence $(a_k)$ satisfying \eqref{eq:standing} for which 
    \begin{equation}\label{main1.1sup}
        \limsup_{n\to\infty}\frac{X(s_n)}{(2v(s_n)\log\log v(s_n))^{1/2}}=\frac{1}{2^{1/2}}\quad\text{{\rm a.s.}}
    \end{equation}
    and
    \begin{equation}\label{main1.1inf}
        \liminf_{n\to\infty}\frac{X(s_n)}{(2v(s_n)\log\log v(s_n))^{1/2}}=-\frac{1}{2^{1/2}}\quad\text{{\rm a.s.}}
    \end{equation}
b) Assume that $\mmp\{\eta=\pm \sigma\}=1/2$. Then, with the same sequence $(a_k)$ and the same $s_n$ as in part (a), 
    \begin{equation}\label{main1.2sup}
        \limsup_{n\to\infty}\frac{X(s_n)}{(v(s_n))^{1/2}}=1\quad\text{{\rm a.s.}}
    \end{equation}
    and
    \begin{equation}\label{main1.2inf}
        \liminf_{n\to\infty}\frac{X(s_n)}{(v(s_n))^{1/2}}=-1\quad\text{{\rm a.s.}}
    \end{equation}
c) Assume that $\eta=\xi-\sigma$, where $\xi$ has an exponential distribution of mean $\sigma$. Then, with the same sequence $(a_k)$ and the same $s_n$ as in part (a), 
\begin{equation}\label{main1.3sup}
\limsup_{n\to\infty}\frac{X(s_n)}{(v(s_n))^{1/2}\log\log v(s_n)}=\frac{1}{2}\quad\text{{\rm a.s.}}
\end{equation}
and
\begin{equation}\label{main1.3inf}
\liminf_{n\to\infty}\frac{X(s_n)}{(v(s_n))^{1/2}}=-1\quad\text{{\rm a.s.}}
\end{equation}
\end{thm}
\begin{thm}\label{thm:main2}
Assume that $\eta$ has a standard normal distribution and set $s_n=2^{-1}\eee^{-n^2}$ for $n\in\mn$. Then there exists a sequence $(a_k)$ satisfying \eqref{eq:standing} for which
    \begin{equation}\label{main2sup}
        \limsup_{n\to\infty}\frac{\log|X(s_n)|-n\log 2}{\log\log n}=\frac{1}{2}\quad\text{{\rm a.s.}}
    \end{equation}
    and
    \begin{equation}\label{main2inf}
        \liminf_{n\to\infty}\frac{\log|X(s_n)|-n\log 2}{\log n}=-1\quad\text{{\rm a.s.}}
    \end{equation}
\end{thm}

In the settings of Theorems~\ref{thm:main1} and \ref{thm:main2}, the $\limsup$ ($\liminf$) is governed by one or a few exceptionally large (small) summands rather than by the cumulative contribution of many small ones. Consequently, the failure of the standard LIL is not surprising. One possible explanation is that the partial sums $\sum_{k\leq n}k^{-1}a_k^2$ exhibit highly irregular growth. It is therefore natural to ask whether imposing a smoothing condition
\begin{equation}\label{eq:basic}
\sum_{k\leq n}\frac{a_k^2}{k}~\sim~c(\log n)^\beta, \quad n\to\infty
\end{equation}
for some $c>0$ and $\beta>0$ is sufficient to restore a standard LIL. At first glance, condition \eqref{eq:basic} appears to rule out the possibility that a single exceptionally large summand dominates $X(s)$. Our next result shows that this intuition is false.
\begin{thm}\label{thm:counterexample3}
Fix $\beta>0$. Assume that $\me [\eta]=0$, ${\rm Var}[\eta]\in (0,\infty)$ and
\begin{equation}\label{eq:tail}
\mmp\{\eta>x\}~\sim~\frac{1}{(x\log x)^2},\quad x\to\infty.
\end{equation}
Then there exists a sequence $(a_k)$ satisfying \eqref{eq:standing} and  \eqref{eq:basic} for which $$\limsup_{s\to 0+}\frac{s^{\beta/2}}{(\log\log 1/s)^{1/2}} X(s)=+\infty\quad\text{{\rm a.s.}}$$
\end{thm}

After a series of `negative' results, we present a `positive' result giving sufficient conditions on $(a_k)$ that ensure a standard LIL. We shall write $\log^{(3)}x$ for $\log\log\log x$. Put $$\mathcal{S}:=\{k\geq 16: 
a_k\neq 0\}.$$ Since we shall normally
work under assumption \eqref{eq:basic}, the set $\mathcal{S}$ is
automatically infinite.
We introduce the following

\noindent {\sc Condition A}. As $x\to\infty$, $$\#\Big\{k\in\mathcal{S}:~\frac{k(\log k)^\beta}{a_k^2(\log\log k)^2\log^{(3)}k}\leq x\Big\}=O(x).$$

\begin{thm}\label{thm:main4}
Assume that $\me [\eta]=0$, $\sigma^2={\rm Var}[\eta]\in (0,\infty)$ and that $(a_k)$ satisfies \eqref{eq:basic} and Condition A. Then
\begin{equation*}
C\Big(\Big(\Big(\frac{1}{(2v(s)\log\log v(s))^{1/2}}\sum_{k\geq 1} \frac{a_k \eta_k}{k^{1/2+s}}: s>0~\text{sufficiently close to}~ 0 \Big)\Big)=[-1,1] \quad\text{\rm a.s.}
\end{equation*}
and $$v(s)~\sim~\frac{\sigma^2 c\Gamma(1+\beta)}{(2s)^\beta},\quad s\to 0+.$$
In particular,
\begin{equation*}
\limsup_{s\to 0+}\frac{1}{(2v(s)\log\log v(s))^{1/2}}\sum_{k\geq 1} \frac{a_k \eta_k}{k^{1/2+s}}=1 \quad\text{\rm a.s.}
\end{equation*}
and
\begin{equation*}
\liminf_{s\to 0+}\frac{1}{(2v(s)\log\log v(s))^{1/2}}\sum_{k\geq 1} \frac{a_k \eta_k}{k^{1/2+s}}=-1 \quad\text{\rm a.s.}
\end{equation*}
\end{thm}
\begin{remark}\label{rem:suff}
A convenient sufficient criterion for Condition~A is the following

\noindent {\sc Condition B}. There exist a positive function $\varphi$ and a constant $C>0$ such that

\noindent (a) $\rho(x):=\#\{k\in\mathcal{S}: k\leq x\}=O(x/\varphi(x))$ as $x\to\infty$;

\noindent (b) $a_k^2 \leq C \frac{(\log k)^\beta \varphi(k)}{(\log\log k)^2\log^{(3)} k}$ for large $k\in\mathcal{S}$.

Sufficiency will be proved at the end of Section \ref{sect:suff}.

In Section \ref{sect:applications}, we provide a number of examples of $(a_k)$ which satisfy Condition~B, hence also Condition~A, as well as an example of $(a_k)$ which satisfies Condition~A but does not satisfy Condition~B.
\end{remark}

\section{Proofs of Theorems \ref{thm:main1}, \ref{thm:main2} and \ref{thm:counterexample3}}\label{sect:main}
\begin{proof}[Proof of Theorem \ref{thm:main1}]
Until further notice we treat all the parts simultaneously and assume that $\eta$ has an arbitrary distribution with mean zero and variance $\sigma^2\in (0,\infty)$. We construct $(a_k)$ to create exactly one dominating summand of the series  defining $X(s)$.

Let $k_j:=\exp(\eee^{j^2})$ for $j\in\mn$ and
\begin{equation}\label{a_k_1}
    a_k:=\begin{cases}(\eee^{j^2}\lfloor k_j\rfloor)^{1/2}, & \quad\text{for}\quad k=\lfloor k_j\rfloor,\quad j\in\mn,\\0, & \quad\text{for other}\quad k.\end{cases}
\end{equation}
Observe for later use that
\begin{equation}\label{eq:ineq}
\lfloor k_j\rfloor>(1-k_j^{-1})k_j\geq (1-\eee^{-\eee})k_j,\quad j\in\mn.
\end{equation}
We first show that $(a_k)$ satisfies \eqref{eq:standing}:
\begin{equation*}
    \sum_k\frac{a_k^2}{k}=\sum_{j\ge 1}\frac{\eee^{j^2}\lfloor k_j\rfloor}{\lfloor k_j\rfloor}=\sum_{j\ge 1}\eee^{j^2}=\infty
\end{equation*}
and, for any $s>0$,
\begin{equation*}
    \sum_k\frac{a_k^2}{k^{1+2s}}=\sum_{j\ge 1}\frac{\eee^{j^2}
    \lfloor k_j\rfloor}{\lfloor k_j\rfloor^{1+2s}}\leq 
    (1-\eee^{-\eee})^{-2s}\sum_{j\ge1}\exp(j^2-2s \eee^{j^2})
    <\infty
\end{equation*}
having utilized \eqref{eq:ineq}.

By Lemma \ref{lemma:main1.Var}, $$\lim_{n\to\infty}\big(\eee^{-n^2/2+s_n\log\lfloor k_n\rfloor}X(s_n)-\eta_{\lfloor k_n\rfloor}\big)=0\quad\text{a.s.}$$ Since $\lim_{n\to\infty}s_n\log \lfloor k_n\rfloor=1/2$, and the sequences $(\eta_n)_{n\geq 1}$ and $(\eta_{\lfloor k_n\rfloor})_{n\geq 1}$ have the same distribution it remains to show that
\begin{equation*}
\limsup_{n\to\infty}\frac{\eta_n}{(2\log n)^{1/2}}=\sigma\quad\text{and}\quad \liminf_{n\to\infty}\frac{\eta_n}{(2\log n)^{1/2}}=-\sigma\quad\text{a.s.}
\end{equation*}
for part (a),
\begin{equation*}
\limsup_{n\to\infty}\eta_n=\sigma\quad\text{and}\quad \liminf_{n\to\infty}\eta_n =-\sigma\quad\text{a.s.}
\end{equation*}
for part (b), and
\begin{equation*}
      \limsup_{n\to\infty}\frac{\eta_n}{\log n}=\sigma\quad\text{and}\quad \liminf_{n\to\infty}\eta_n =-\sigma \quad\text{a.s.}
\end{equation*}
for part (c).
The proof is a standard exercise on an application of the Borel-Cantelli lemma, see, for instance, p.~109 in \cite{Gut:2005} for the case where $\eta$ has a normal distribution, and we omit details.

By Lemma \ref{lemma:main1.Var}, $v(s_n)\sim\sigma^2 \eee^{n^2-1}$ as $n\to\infty$. Hence, the normalizations in parts (a), (b) and (c) are indeed as stated in the theorem.
\end{proof}

\begin{proof}[Proof of Theorem \ref{thm:main2}]
We construct $(a_k)$ to create a group of dominating summands of the series  defining $X(s)$ and a group of negligible summands.

Let $k_j:=\exp(\eee^{j^2})$ and
\begin{equation}\label{a_k_2}
    a_k:=\begin{cases}(4^j\lfloor k_j\rfloor)^{1/2}, & \quad\text{for}\quad k=\lfloor k_j\rfloor,\quad j\in\mn,\\0, &\quad\text{for other}\quad k.\end{cases}
\end{equation}
The $(a_k)$ so defined satisfies \eqref{eq:standing} as follows from
\begin{equation*}
    \sum_k\frac{a_k^2}{k}=\sum_{j\ge 1}\frac{4^j\lfloor k_j\rfloor}{\lfloor k_j\rfloor}=\sum_{j\ge 1}4^j=\infty
\end{equation*}
and, for any $s>0$,
\begin{equation*}
    \sum_k\frac{a_k^2}{k^{1+2s}}=\sum_{j\ge 1}\frac{4^j\lfloor k_j\rfloor}{\lfloor k_j\rfloor^{1+2s}}\leq (1-\eee^{-\eee})^{-2s} \sum_{j\ge1}\frac{4^j}{(\exp(\eee^{j^2}))^{2s}}<\infty.
\end{equation*}
We have used \eqref{eq:ineq} for the penultimate inequality.

Put $\bar X(s):=\sum_{j\ge 1}k_j^{-s}2^j \eta_j$ for $s>0$. We first prove the theorem with $$\bar X(s_n)=\sum_{j\ge 1}2^j\exp(-\eee^{j^2-n^2}/2)\eta_j$$ replacing $X(s_n)$. Put
\begin{equation*}
Y_n:=\sum_{j=1}^n 2^j\exp\big(-\eee^{j^2-n^2}/2\big)\eta_j\quad \text{and}\quad Z_n:=\sum_{j\ge n+1}2^j\exp\big(-\eee^{j^2-n^2}/2\big)\eta_j,\quad n\in\mn.
\end{equation*}
Then $$\log |\bar X(s_n)2^{-n}|=\log |Y_n 2^{-n}|+\log \Big|1+\frac{Z_n}{Y_n}\Big|.$$ By Corollary \ref{lemma:main2.D}, relations \eqref{main2.Dsup} and \eqref{main2.Dinf} hold,
whereas, by Corollary \ref{lemma:main2.CD}, $\log \Big|1+\frac{Z_n}{Y_n}\Big|=o(1)$ as $n\to\infty$ a.s. Hence,
\begin{equation*}
\limsup_{n\to\infty}\frac{\log|\bar X(s_n)|-n\log 2}{\log\log n}=\frac{1}{2}\quad\text{{\rm a.s.}}
\end{equation*}
and
\begin{equation*}
\liminf_{n\to\infty}\frac{\log|\bar X(s_n)|-n\log 2}{\log n}=-1\quad\text{{\rm a.s.}}
\end{equation*}
The sequences $(\bar X(s_n))_{n\geq 1}$ and $(X_1(s_n))$ have the same distribution. Hence, the two latter asymptotic relations hold with $X_1(s_n)$ replacing $\bar X(s_n)$. This entails $\lim_{n\to\infty}n^{-1}\log |X_1(s_n)|=\log 2$ a.s. To show that the same limit relations also hold for $X(s_n)$, observe that, by Lemma \ref{lem:reduction},
$$\lim_{n\to\infty}(X(s_n)-X_1(s_n))=0\quad\text{a.s.}$$ Together with the fact that $|X_1(s_n)|$ diverges to $\infty$ this guarantees $\lim_{n\to\infty}|X(s_n)/X_1(s_n)|=1$ a.s., whence $\lim_{n\to\infty}(\log |X(s_n)|-\log |X_1(s_n)|)=0$ a.s. The proof of Theorem \ref{thm:main2} is complete.
\end{proof}

\begin{proof}[Proof of Theorem \ref{thm:counterexample3}]
We construct $(a_k)$ by introducing a sparse {\it perturbation set} $(k_j)_{j\ge1}$ on which the quantities $k^{-1}a_k^2$ are exceptionally large, thereby violating the uniform smallness of the summands with respect to the variance. The perturbation set is chosen sufficiently sparse to ensure that the asymptotic relation \eqref{eq:basic} is preserved.

Guided by this idea, for $j\in\mn$, put $s_j:=\exp(j^2)$ and $k_j:=\lfloor \exp(s_j)\rfloor=\lfloor \exp(\eee^{j^2})\rfloor$. Now define $a_1=a_{15}=:=0$, $$\frac{a_k^2}{k}:=
\begin{cases}
        \frac{s_j^\beta}{(\log\log s_j)^{1/2}}, &   \text{if} \ k=k_j,~j\geq 2,   \\
        \frac{(\log k)^{\beta-1}}{k}, & \text{if} \ k\neq k_j,~j\in\mn,~k\geq 2.
\end{cases}$$
We first verify that \eqref{eq:basic} holds with this choice of $(a_k)$. To this end, put $m_n:=\max\{j\in\mn: k_j\leq n\}$ for $n\in\mn$, so that $$\sum_{k_j\leq n}\frac{a_{k_j}^2}{k_j}=\sum_{j\leq m_n}\frac{s_j^\beta}{(\log\log s_j)^{1/2}}.$$ Since $s_j^\beta=\exp(\beta j^2)$ satisfies $\lim_{j\to\infty}(s_j/s_{j-1})^\beta=\infty$, the last sum is dominated by the term with $j=m_n$ in the sense that $$\sum_{j\leq m_n}\frac{s_j^{\beta}}{(\log\log s_j)^{1/2}}~\sim~ \frac{s_{m_n}^{\beta}}{(\log\log s_{m_n})^{1/2}},\quad n\to\infty.$$ Further, $s_{m_n}=\exp(m_n^2)\leq \log (n+1)$ and thereupon $$\frac{s_{m_n}^{\beta}}{(\log\log s_{m_n})^{1/2}}\leq \frac{(\log(n+1))^{\beta}}{(\log\log\log (n+1))^{1/2}} =o((\log n)^{\beta}),\quad n\to\infty$$ because the function $x\mapsto x^\beta(\log\log x)^{-1/2}$ is increasing for large $x$. Combining the estimates yields
\begin{equation}\label{eq:basic22}
\sum_{k_j\leq n}\frac{a_{k_j}^2}{k_j}=o((\log n)^{\beta}),\quad n\to\infty.
\end{equation}
Since $j\mapsto k_j^{-1}(\log k_j)^{\beta-1}$ decays superexponentially fast, $$\sum_{k_j\leq n}\frac{(\log k_j)^{\beta-1}}{k_j}=\sum_{j\leq m_n}\frac{(\log k_j)^{\beta-1}}{k_j}=O(1),\quad n\to\infty.$$ This in combination with \eqref{eq:basic2} proves $$\sum_{k\leq n,\, k\neq k_j}\frac{a_k^2}{k}~\sim~\frac{1}{\beta}(\log n)^{\beta},\quad n\to\infty.$$ Recalling \eqref{eq:basic22} we arrive at \eqref{eq:basic}. In particular, the first condition in \eqref{eq:standing} holds.

For each $s>0$,
$$\sum_{k\geq1}\frac{a_k^2}{k^{1+2s}}\leq \sum_{k\geq2}\frac{(\log k)^{\beta-1}}{k^{1+2s}}+\sum_{j\geq 1}\frac{s_j^\beta k_j^{-2s}}{(\log\log s_j)^{1/2}}<\infty.$$
Thus, the second condition in \eqref{eq:standing} also holds.

Next, we decompose $X(s)$ into its contributions from the perturbation set and the complementary {\it normal set} $\{k\in\mn: k\neq k_j,~j\in\mn\}$:
\begin{multline*}
\frac{s^{\beta/2}}{(\log\log 1/s)^{1/2}}X(s)=\frac{s^{\beta/2}}{(\log\log 1/s)^{1/2}}\Big(\sum_{j\geq 1}\eee^{-s\log k_j}\frac{s_j^{\beta/2}\eta_{k_j}}{(\log\log s_j)^{1/4}}+\sum_{k\neq k_j}\frac{(\log k)^{(\beta-1)/2}}{k^{1/2+s}}\eta_k\Big)\\=:Y_1(s)+Y_2(s).
\end{multline*}
Put $s=s_n^{-1}$ and note that $\lim_{n\to\infty} s_n^{-1}\log k_n=1$. Thus, the contribution of the $n$th term in $Y_1(s_n^{-1})$ is $$\sim \eee^{-1}\frac{s_n^{\beta/2}\eta_{k_n}}{(\log\log s_n)^{1/4}}\frac{1}{s_n^{\beta/2}(\log\log s_n)^{1/2}}=\frac{\eee^{-1}\eta_{k_n}}{(2\log n)^{3/4}}.$$ Using \eqref{eq:tail} in combination with the Borel-Cantelli lemma we conclude that $$\limsup_{n\to\infty}\frac{\eta_{k_n}}{(\log n)^{3/4}}=+\infty\quad\text{a.s.}$$ It therefore remains to show that neither $Y_2(s_n^{-1})$ nor the sum of the remaining terms of $Y_1(s_n^{-1})$ can diverge to $-\infty$ almost surely.

By Proposition \ref{prop:bura} with $\alpha=(\beta-1)/2$, $$s\mapsto \frac{s^{\beta/2}}{(\log\log 1/s)^{1/2}}\sum_{k\geq 2}\frac{(\log k)^{(\beta-1)/2}}{k^{1/2+s}}\eta_k$$ remains bounded as $s\to 0+$. The same is true of $Y_2$. Indeed, $$\lim_{s\to 0+}\sum_{j\geq 1}\frac{(\log k_j)^{(\beta-1)/2}}{k_j^{1/2+s}}\eta_{k_j}=\sum_{j\geq 1}\frac{(\log k_j)^{(\beta-1)/2}}{k_j^{1/2}}\eta_{k_j}\quad\text{a.s.},$$ and, in view of $\sum_{j\geq 1}k_j^{-1} (\log k_j)^{\beta-1}<\infty$, the limit random variable is a.s.\ finite.

It remains to estimate the variance of the remaining perturbation terms: $${\rm Var}\Big[\sum_{j\neq n}\eee^{-s_n^{-1}\log k_j}\frac{s_j^{\beta/2}\eta_{k_j}}{(\log\log s_j)^{1/4}}\Big]=\sigma^2\sum_{j\neq n}\eee^{-2s_n^{-1}\log k_j}\frac{s_j^{\beta}}{(\log\log s_j)^{1/2}}.$$ Using $\lim_{n\to\infty}\max_{1\leq j\leq n}\,\eee^{-2s_n^{-1}\log k_j}=1$ for the first inequality we infer
\begin{multline*}
\sum_{j\leq n-1}\eee^{-2s_n^{-1}\log k_j}\frac{s_j^{\beta}}{(\log\log s_j)^{1/2}}\leq 2 \sum_{j\leq n-1}\frac{s_j^{\beta}}{(\log\log s_j)^{1/2}}~\sim~\frac{2s_{n-1}^{\beta}}{(\log\log s_{n-1})^{1/2}}, \quad n\to\infty.
\end{multline*}
The remaining sum is dominated by the $(n+1)$st term which is $$\sim \sigma^2 \exp(-2\eee^{2n+1}+\beta(n+1)^2)(2\log(n+1))^{-1/2}=o(1),\quad n\to\infty.$$ Consequently, $${\rm Var}\Big[\sum_{j\neq n}\eee^{-s_n^{-1}\log k_j}\frac{s_j^{\beta/2}\eta_{k_j}}{(\log\log s_j)^{1/4}}\Big]=O\Big(\frac{s_{n-1}^{\beta}}{(\log\log s_{n-1})^{1/2}}\Big),\quad n\to\infty.$$ Since $((s_{n-1}/s_n)^{\beta})=(\eee^{-\beta(2n-1)})$ is a summable sequence, we conclude that $$\lim_{n\to\infty}\frac{1}{s_n^{\beta/2}(\log\log s_n)^{1/2}} \sum_{j\neq n}\eee^{-s_n^{-1}\log k_j}\frac{s_j^{\beta/2}\eta_{k_j}}{(\log\log s_j)^{1/4}}=0\quad \text{a.s.}$$ having utilized Markov's inequality and the direct part of the Borel-Cantelli lemma.

The proof of Theorem \ref{thm:counterexample3} is complete.
\end{proof}

\section{Supporting tools for Section \ref{sect:main}}

Lemma \ref{lemma:main1.Var} is used in the proof of Theorem \ref{thm:main1}.
\begin{lemma}\label{lemma:main1.Var}
Let $(a_k)$ be as defined in \eqref{a_k_1}, $s_n=2^{-1}\eee^{-n^2}$ for $n\in\mn$. Assume that $\eta$ has an arbitrary distribution with mean zero and variance $\sigma^2\in (0,\infty)$. Then
    \begin{equation}\label{main1.Var}
        v(s_n)={\rm Var}[X(s_n)] 
        ~\sim~\sigma^2\eee^{n^2-1},\quad n\to\infty
    \end{equation}
and 
\begin{equation}\label{eq:asconv}
\lim_{n\to\infty}\big(\eee^{-n^2/2+s_n\log\lfloor k_n\rfloor}X(s_n)-\eta_{\lfloor k_n\rfloor}\big)=0,\quad \text{{\rm a.s.}}
\end{equation}
\end{lemma}
\begin{proof}
In view of $$\sum_{j\geq 1}\frac{\eee^{j^2}}{k_j^{2s_n}}\leq \sigma^{-2}v(s_n)=\sum_{j\geq 1}\frac{\eee^{j^2}}{\lfloor k_j\rfloor^{2s_n}}\leq (1-\eee^{-\eee})^{-2s_n}\sum_{j\geq 1}\frac{\eee^{j^2}}{k_j^{2s_n}},$$ where the right-hand side is secured by \eqref{eq:ineq}, it suffices to prove that
$$\sum_{j\geq 1}\frac{\eee^{j^2}}{k_j^{2s_n}}=\sum_{j\geq 1}\exp(j^2-\eee^{j^2-n^2})~\sim~ \eee^{n^2-1},\quad n\to\infty.$$ Since the right-hand side is equal to the $n$th term of the series, the task boils down to showing that
\begin{equation}\label{main1.Var1}
\lim_{n\to\infty}\sum_{j\ge1,\,j\neq n}\exp(j^2-n^2-\eee^{j^2-n^2})=0.
\end{equation}
By monotonicity, $$\sum_{j=1}^{n-1}\exp(j^2-n^2-\eee^{j^2-n^2})\leq \sum_{j=1}^{n-1}\eee^{j^2-n^2}\leq (n-1)\eee^{(n-1)^2-n^2}=(n-1)\eee^{-2n+1}~\to~0,~ n\to\infty.$$ Using $\eee^x\geq 1+x+x^2/2$ which holds for $x\in\mr$ we infer $$(n+k)^2-\eee^{(n+k)^2-n^2}\leq n^2-1-k^4/2-2nk^3-2n^2k^2\leq -2nk^3\leq -2nk$$ for $k,n\in\mn$ and thereupon
\begin{multline*}
      \sum_{j\ge n+1}\exp(j^2-\eee^{j^2-n^2})=\sum_{k\ge1}\exp((n+k)^2-\eee^{(n+k)^2-n^2})\le \sum_{k\ge 1}\eee^{-2nk}~\sim~\eee^{-2n}\to 0,~~n\to\infty.
  \end{multline*}
This proves \eqref{main1.Var1}, hence also \eqref{main1.Var}.

As for \eqref{eq:asconv}, note that the $n$th term of the series defining $X(s_n)$ is equal to $\eee^{n^2/2-s_n\log\lfloor k_n\rfloor}\eta_{\lfloor k_n\rfloor}$. Hence, it is enough to prove that
\begin{equation}\label{eq:asconv1}
\lim_{n\to\infty}\eee^{-n^2/2+s_n\log\lfloor k_n\rfloor} \sum_{j\ge1,\,j\neq n}\frac{\eee^{j^2/2}}{\lfloor k_j\rfloor^{s_n}}\eta_{\lfloor k_j\rfloor}=0\quad\text{a.s.}
\end{equation}
Using $\eee^{2s_n\log\lfloor k_n\rfloor}\leq \eee$ and $\lfloor k_j\rfloor^{-2s_n}
\leq (1-\eee^{-\eee})^{-2s_n}k_j^{-2s_n}$ which follows from \eqref{eq:ineq} we conclude that  
\begin{multline*}
\me\Big[\Big(\eee^{-n^2/2+s_n\log\lfloor k_n\rfloor} \sum_{j\ge1,\,j\neq n}\frac{\eee^{j^2/2}}{\lfloor k_j\rfloor^{s_n}}
\eta_{\lfloor k_j\rfloor}\Big)^2\Big]=\sigma^2
\eee^{-n^2+2s_n\log\lfloor k_n\rfloor}
\sum_{j\geq1,\,j\neq n}\frac{\eee^{j^2}}{\lfloor k_j\rfloor^{2s_n}}\\\leq C\sum_{j\geq1,\,j\neq n}\exp\big(j^2-n^2-\eee^{j^2-n^2}\big)
\end{multline*}
for some constant $C>0$ and sufficiently large $n$. According to Markov's inequality, for all $\varepsilon>0$,
\begin{multline*}
\mmp\Big\{\eee^{-n^2/2+s_n\log\lfloor k_n\rfloor}\Big|\sum_{j\ge1,\,j\neq n}\frac{\eee^{j^2/2}}{\lfloor k_j\rfloor^{s_n}}\eta_{\lfloor k_j\rfloor}\Big| >\varepsilon\Big\}\le \varepsilon^{-2}C \sum_{j\ge1,\,j\neq n}\exp(j^2-n^2-\eee^{j^2-n^2}). 
\end{multline*}
It follows from the previous part of the proof that the right-hand side is the $n$th term of a summable sequence. Now \eqref{eq:asconv1}, hence also \eqref{eq:asconv}, is secured by the direct part of the Borel-Cantelli lemma.
\end{proof}

All the subsequent results of this section are used in the proof of Theorem \ref{thm:main2}.
\begin{lemma}\label{lem:reduction}
Let $(a_k)$ be as defined in \eqref{a_k_2}, $s_n=2^{-1}\eee^{-n^2}$ for $n\in\mn$. Assume that $\eta$ has an arbitrary distribution with mean zero and variance $\sigma^2\in (0,\infty)$. Then
$\lim_{n\to\infty}(X(s_n)-X_1(s_n))=0$ a.s., where $X_1(s):=\sum_{j\geq 1}k_j^{-s}2^j \eta_{\lfloor k_j\rfloor}$ for $s>0$.

\end{lemma}
\begin{proof}
An application of the mean value theorem for differentiable functions to $x\mapsto x^{-s_n}$ yields $$0\leq\lfloor k_j\rfloor^{-s_n}-k_j^{-s_n}\leq s_n \lfloor k_j\rfloor^{-s_n-1}.$$ This ensures that $$\me [|X(s_n)-X_1(s_n)|]\leq \me [|\eta|]s_n\sum_{j\geq 1}\lfloor k_j\rfloor^{-s_n-1}2^j ~\sim~\me [|\eta|]\sum_{j\geq 1}\lfloor k_j\rfloor^{-1}2^j s_n,\quad n\to\infty.$$ As a consequence, $\sum_{n\geq 1}\me[|X(s_n)-X_1(s_n)|]<\infty$ which proves $\lim_{n\to\infty}(X(s_n)- X_1(s_n))=0$ a.s. 
\end{proof}
\begin{lemma}\label{lemma:main2.Var}
The following asymptotic relations hold:
\begin{equation}\label{main2.Var1}
    \sum_{j=1}^n 4^j\exp(-\eee^{j^2-n^2})~\sim~\Big(\frac{1}{\eee}+\frac{1}{3}\Big)4^n,~~n\to\infty
\end{equation}
and
\begin{equation}\label{main2.Var2}
\sum_{j\ge n+1} 4^j\exp(-\eee^{j^2-n^2})~\sim~ 4^{n+1}\exp(-\eee^{2n+1}),~~n\to\infty.
\end{equation}
\end{lemma}
\begin{proof}
Since $\lim_{n\to\infty}\exp(-\eee^{j^2-n^2})=1$ uniformly in $j\in \{1,2,\ldots, n-1\}$ we conclude that
\begin{multline*}
\lim_{n\to\infty}\sum_{j=1}^n 4^{j-n}\exp(-\eee^{j^2-n^2})=\eee^{-1}+\lim_{n\to\infty}\sum_{j=1}^{n-1} 4^{j-n}\exp(-\eee^{j^2-n^2})=\eee^{-1}+\lim_{n\to\infty}\sum_{j=1}^{n-1} 4^{j-n}\\=\eee^{-1}+3^{-1}.
\end{multline*}
To prove \eqref{main2.Var2}, write
\begin{multline*}
\lim_{n\to\infty}\sum_{j\ge n+1} 4^{j-n-1}\exp(\eee^{2n+1}-\eee^{j^2-n^2})=1+\lim_{n\to\infty}\sum_{j\ge n+2} 4^{j-n-1}\exp(\eee^{2n+1}-\eee^{j^2-n^2})\\=1+\lim_{n\to\infty}\sum_{k\geq 1} 4^k\exp(\eee^{2n+1}-\eee^{(n+k+1)^2-n^2})=1. 
\end{multline*}
The last equality is justified  by the Lebesgue dominating convergence theorem upon noting that, for each $k\in\mn$, $\lim_{n\to\infty} (\eee^{2n+1}-\eee^{(n+k+1)^2-n^2})=-\infty$ and that, for $k,n\in\mn$, $\eee^{2n+1}-\eee^{(n+k+1)^2-n^2}\leq \eee^3-\eee^{k^2+4k+3}$. The latter inequality holds because, for each fixed $k\in\mn$, the sequence $n\mapsto \eee^{2n+1}-\eee^{(n+k+1)^2-n^2}$ decreases on $\mn$.
\end{proof}

\begin{lemma}\label{lemma:main2.C}
As $n\to\infty$,
\begin{equation}\label{main2.C}
Z_n=\sum_{j\ge n+1}2^j\exp\big(-\eee^{j^2-n^2}/2\big)\eta_j=o\big(\exp\big(-\eee^{2n}/2\big)\big)\quad\text{{\rm a.s.}}
\end{equation}
\end{lemma}
\begin{proof}
By Lemma \ref{lemma:main2.Var}, 
\begin{equation*}
\text{Var}[Z_n]=\sigma^2 \sum_{j\ge n+1} 4^j\exp(-\eee^{j^2-n^2})~\sim~\sigma^2 4^{n+1}\exp(-\eee^{2n+1}),\quad n\to\infty.
\end{equation*}
Using this we obtain with the help of Markov's inequality that, for all $\varepsilon>0$,
\begin{equation*}
\mmp\big\{|Z_n|>\varepsilon\exp\big(-\eee^{2n}/2\big)\big\}\le 
\varepsilon^{-2}\exp(\eee^{2n}) \text{Var}[Z_n]~\sim~ \varepsilon^{-2}\sigma^2 4^{n+1}\exp(\eee^{2n}-\eee^{2n+1}),\quad n\to\infty.  
\end{equation*}
Since the right-hand side is the $n$th term of a summable sequence, the direct part of the Borel-Cantelli lemma secures the claim.
\end{proof}

Let $\eta_0$, $\eta_{-1},\ldots$ be independent identically distributed random variables with the same distribution as $\eta$, which are independent of $\eta_1,\eta_2,\ldots$. Put
\begin{equation*}
    \hat{Z}_n=\eee^{-1/2}\eta_n+\frac{1}{2}\eta_{n-1}+\cdots +\frac{1}{2^{n-1}}\eta_1+\frac{1}{2^{n}}\eta_0+\frac{1}{2^{n+1}}\eta_{-1}+\cdots,\quad n\in\mathbb{Z}.
\end{equation*}
\begin{lemma}\label{lemma:main2.Ginf}
Assume that $\eta$ has a standard normal distribution. Then
\begin{equation}\label{main2Ginf}
\liminf_{n\to\infty}\frac{\log|\hat{Z}_n|}{\log n}=-1~~\text{{\rm a.s}}.
\end{equation}
\end{lemma}
\begin{proof}
The sequence $(\hat{Z}_n)_{n\in\mathbb{Z}}$ is stationary. In particular. $\hat{Z}_n\overset{{\rm d}}{=}\hat{Z}_0$ for each $n\in\mathbb{Z}$. Plainly, $\hat{Z}_0$ has a normal distribution with mean zero and variance 
\begin{equation*}
a^2=(\eee^{-1/2})^2\text{Var}[\eta_0]+4^{-1}\text{Var}[\eta_{-1}]+4^{-2}\text{Var}[\eta_{-2}]+\ldots=\eee^{-1}+3^{-1}. 
\end{equation*}
Further, we claim that $$\rho_k:=\frac{\text{Cov}(\hat{Z}_k,\hat{Z}_0)}{a^2}=\frac{\eee^{-1/2}+3^{-1}}{\eee^{-1}+3^{-1}}2^{-k}=:c2^{-k},\quad k\in\mn.$$ Indeed, if $k\in\mn$,
\begin{multline*}
\text{Cov}(\hat{Z}_k,\hat{Z}_0)=\me[\hat{Z}_k\hat{Z}_0]=\me[\hat{Z}_k\hat{Z}_0]\\=\me\big[(\eee^{-1/2}\eta_k+2^{-1}\eta_{k-1}+\ldots+2^{-k}\eta_0+2^{-k-1}\eta_{-1}+\ldots)(\eee^{-1/2}\eta_0+2^{-1}\eta_{-1}
+\ldots)\big]\\=\eee^{-1/2}2^{-k}+2^{-1} 2^{-k-1}+2^{-2}2^{-k-2}+\ldots=\big(\eee^{-1/2}+4^{-1}+4^{-2}+\ldots\big)2^{-k}=\big(\eee^{-1/2}+3^{-1}\big)2^{-k}.
\end{multline*}

By L'Hôpital's rule,
\begin{equation*}
\mmp\{|\hat{Z}_0|\le x\}=\frac{1}{(2\pi)^{1/2}a}\int_{-x}^x\eee^{-\frac{y^2}{2a^2}}{\rm d}y~\sim~ \frac{2^{1/2}}{\pi^{1/2}a}x=:c_a x,\quad x\to 0+.
    \end{equation*}
Hence, for all $\varepsilon>0$, 
\begin{equation*}
\mmp\{\log|\hat{Z}_n|\le-(1+\varepsilon)\log n\}=
\mmp\{|\hat{Z}_0|\le n^{-(1+\varepsilon)}\}~\sim~c_a n^{-(1+\varepsilon)},\quad n\to\infty.
\end{equation*}
Therefore, $\log|\hat{Z}_n|\ge-(1+\varepsilon)\log n$ for large $n$ a.s.\ according to the direct part of the Borel-Cantelli lemma, that is,
\begin{equation}\label{main2Ginf1}
\liminf_{n\to\infty}\frac{\log|\hat{Z}_n|}{\log n}\ge-1\quad \text{{\rm a.s.}}
\end{equation}
Put $A_n:=\{\log|\hat{Z}_n|\le-(1-\varepsilon)\log n\}$ for $\varepsilon\in(0,1/2)$. We want to show that $\mmp\{A_n~\text{i.o.}\}=1$. The events $A_1,A_2,\cdots$ are dependent, so the converse part of the Borel-Cantelli lemma is not applicable. However, by the Kochen-Stone lemma (see, for instance, p.73 in \cite{Durrett}), if $\sum_{n\ge1}\mmp(A_n)=\infty$, then
\begin{equation*}
\mmp\{A_n~\text{i.o.}\}\ge\limsup_{n\to\infty}\frac{(\sum_{k=1}^n \mmp(A_k))^2}{\sum_{i=1}^n\sum_{j=1}^n\mmp(A_i\cap A_j)}
\end{equation*}
Using $$\mmp(A_n)=\mmp\{|\hat{Z}_n|\le n^{\varepsilon-1}\}~\sim~c_a n^{\varepsilon-1},\quad n\to\infty$$ we infer $\sum_{n\ge 1}\mmp(A_n)=\infty$.
   
    For $i<j$,
    \begin{equation}\label{main2H}
        \mmp(A_i\cap A_j)=\mmp\{|\hat{Z}_i|\le i^{\varepsilon-1},~|\hat{Z}_j|\le j^{\varepsilon-1}\}=\int_{|x|\le i^{\varepsilon-1}}\int_{|y|\le j^{\varepsilon-1}}f_{j-i}(x,y){\rm d}y{\rm d}x,
    \end{equation}
    where $f_{j-i}$ is the density of $(\hat{Z}_i,\hat{Z}_j)$. It is known that the density of a vector $(X_1,X_2)$ with a two-dimensional normal distribution is given by
    \begin{equation*}
        f(x,y)=\frac{1}{2\pi\sigma_1\sigma_2(1-\rho^2)^{1/2}}\cdot\exp\Big(-\frac{1}{2(1-\rho^2)}\Big(\frac{(x-\mu_1)^2}{\sigma_1^2}+\frac{(y-\mu_2)^2}{\sigma_2^2}-\frac{2\rho(x-\mu_1)(y-\mu_2)}{\sigma_1\sigma_2}\Big)\Big),
    \end{equation*}
where $\mu_i=\me [X_i]$, $\sigma^2_i={\rm Var}[X_i]$ for $i=1,2$ and $\rho$ is the correlation coefficient defined by $\rho=(\sigma_1\sigma_2)^{-1}(\me [X_1X_2]-\mu_1\mu_2)$.
   Substituting $\mu_1=\mu_2=0$ and $\sigma_1=\sigma_2=a$ we obtain
   \begin{equation*}
       f_{j-i}(x,y)=
       \frac{1}{2\pi a^2(1-\rho^2)^{1/2}}\exp\Big(-\frac{x^2-2\rho xy+y^2}{2a^2(1-\rho^2)}\Big),\quad x,y\in\mr,
   \end{equation*}
where $\rho=\rho_{j-i}$. Put $\rho_0:=0$ and $$f_0(x,y):= \frac{1}{2\pi a^2}\exp\Big(-\frac{x^2+y^2}{2a^2}\Big),\quad x,y\in\mr.$$ In view of $2\rho xy-\rho^2(x^2+y^2)\le(\rho-\rho^2)(x^2+y^2)\le 2\rho(1-\rho)$ for $x,y\in[-1,1]$,
\begin{multline*}
\frac{f_{j-i}(x,y)}{f_0(x,y)}=\frac{1}{(1-\rho^2)^{1/2}}\exp\Big(\frac{2\rho xy-\rho^2(x^2+y^2)}{2a^2(1-\rho^2)}\Big)\le\frac{1}{(1-\rho^2)^{1/2}}\exp\Big(\frac{\rho}{a^2(1+\rho)}\Big)\\ \le\frac{1}{(1-\rho^2)^{1/2}}\exp\Big(\frac{\rho}{a^2}\Big),\quad x,y\in [-1,1].
   \end{multline*}
As a consequence of
\begin{equation*}
\rho_k=\frac{\eee^{-1/2}+3^{-1}}{\eee^{-1}+3^{-1}}2^{-k}<\frac{3}{2}\cdot\frac{1}{2}=\frac{3}{4},~~k\in\mn,
\end{equation*}
there exist positive constants $c_1, c_2$ and $c_3$ such that
   \begin{equation*}
       \exp\Big(\frac{\rho}{a^2}\Big)\le 1+c_1\rho;~~~\frac{1}{(1-\rho^2)^{1/2}}\le 1+c_2\rho^2\le1+c_2\rho,
   \end{equation*}
and thereupon
\begin{equation*}
       \frac{1}{(1-\rho^2)^{1/2}}\exp\Big(\frac{\rho}{a^2}\Big)\le 1+c_3\rho.
   \end{equation*}
According to \eqref{main2H}, for $i<j$,
\begin{equation*}
\mmp(A_i\cap A_j)\le (1+c_3\rho_{j-i})\int_{|x|\le i^{\varepsilon-1}}\int_{|y|\le j^{\varepsilon-1}}f_0(x,y){\rm d}y{\rm d}x=(1+c_3\rho_{j-i})\mmp(A_i)\mmp(A_j).
\end{equation*}
With this at hand, we proceed as follows:
\begin{multline}\label{main2G1}
\sum_{i=1}^n\sum_{j=1}^n\mmp(A_i\cap A_j)=\sum_{i=1}^n\mmp(A_i)+2\sum_{1\le i<j\le n}\mmp(A_i\cap A_j)\\=\sum_{i=1}^n(\mmp(A_i))^2+2\sum_{1\le i<j\le n}\mmp(A_i\cap A_j)+O\Big(\sum_{i=1}^n\mmp(A_i)\Big)\\ \le \sum_{i=1}^n(\mmp(A_i))^2+2\sum_{1\le i<j\le n}\mmp(A_i)\mmp (A_j)+2c_3\sum_{1\le i<j\le n}\rho_{j-i}\mmp(A_i)\mmp (A_j)+O\Big(\sum_{i=1}^n\mmp(A_i)\Big)\\=\Big(\sum_{i=1}^n\mmp(A_i)\Big)^2+2c_3\sum_{1\le i<j\le n}\rho_{j-i}\mmp(A_i)\mmp (A_j)+O\Big(\sum_{i=1}^n\mmp(A_i)\Big)
   \end{multline}
We intend to show that
   \begin{equation}\label{main2G2}
       \sum_{1\le i<j\le n}\rho_{j-i}\mmp(A_i)\mmp (A_j)=O(1),\quad n\to\infty.
   \end{equation}
To this end, note that, for $i\in\mn$,
\begin{equation*}
\mmp(A_i)=\mmp\{|\hat{Z}_i|\le i^{\varepsilon-1}\}=\frac{1}{(2\pi)^{1/2}a}\int_{-i^{\varepsilon-1}}^{i^{\varepsilon-1}}\eee^{-\frac{x^2}{2a^2}}{\rm d}x\le c_a i^{\varepsilon-1}. 
   \end{equation*}

Therefore, recalling that $\varepsilon\in (0,1/2)$ to ensure the convergence of the resulting series we infer
\begin{multline*}
       \sum_{1\le i<j\le n}\rho_{j-i}\mmp(A_i)\mmp (A_j)\le cc_a^2 \sum_{1\le i<j\le n}2^{i-j} i^{\varepsilon-1}j^{\varepsilon-1}=cc_a^2\sum_{i=1}^n\sum_{j=i+1}^n 2^{i-j}i^{\varepsilon-1}j^{\varepsilon-1}\\ \le cc_a^2\sum_{i=1}^n i^{2\varepsilon-2}\sum_{j=i+1}^n 2^{i-j}\leq cc_a^2\sum_{i\geq 1}i^{2\varepsilon-2}<\infty,
   \end{multline*}
and \eqref{main2G2} follows.

Finally,
   \begin{equation*}
       \mmp\{A_n~\text{i.o.}\}\ge\limsup_{n\to\infty}\frac{(\sum_{k=1}^n \mmp(A_k))^2}{\sum_{i=1}^n\sum_{j=1}^n\mmp(A_i\cap A_j)}\ge 1,
   \end{equation*}
   which entails
   \begin{equation}\label{main2Ginf2}
         \liminf_{n\to\infty}\frac{\log|\hat{Z}_n|}{\log n}\le-1~~\text{{\rm a.s.}}
    \end{equation}
A combination of \eqref{main2Ginf1} and \eqref{main2Ginf2} yields \eqref{main2Ginf}.
\end{proof}

\begin{lemma}\label{lemma:main2.Gsup}
Assume that $\eta$ has a standard normal distribution. Then
    \begin{equation}\label{main2Gsup}
        \limsup_{n\to\infty}\frac{\log|\hat{Z}_n|}{\log\log n}=\frac{1}{2}\quad \text{{\rm a.s}}.
    \end{equation}
\end{lemma}
\begin{proof}
The limit relation
\begin{equation}\label{main2Gsup1}
\limsup_{n\to\infty}\frac{\log|\hat{Z}_n|}{\log\log n}\le\frac{1}{2}~~\text{a.s.}
\end{equation}
follows by a standard application of the direct part of the Borel-Cantelli lemma which is based on
\begin{multline*}
\mmp\{\log |\hat Z_n|>(1/2+\varepsilon)\log\log n\}=\mmp\{|\hat{Z}_n|>(\log n)^{1/2+\varepsilon}\}\\~\sim~
{\rm const} \frac{1}{(\log n)^{1/2+\varepsilon}}\exp\Big(-\frac{(\log n)^{1+2\varepsilon}}{2a^2}\Big),\quad n\to\infty
\end{multline*}
which holds for all $\varepsilon>0$ and the fact that the right-hand side is the $n$th term of a summable sequence.

Since both $(\hat{Z}_k)_{k\in\mathbb{Z}}$ and $(-\hat{Z}_k)_{k\in\mathbb{Z}}$ are stationary Gaussian sequences with $$\rho_k=a^{-2}{\rm Cov}(\hat Z_k, \hat Z_0)=c2^{-k}\to 0, \quad k\to\infty,$$ an application of Theorem 3.4 in \cite{Pickands} yields
    \begin{equation*}
        \lim_{n\to\infty}\frac{\max_{1\le k \le n}\hat{Z}_k}{(2\log n)^{1/2}}=\lim_{n\to\infty}\frac{\max_{1\le k \le n}(-\hat{Z}_k)}{(2\log n)^{1/2}}=a\quad\text{a.s.}
    \end{equation*}

In view of $\max_{1\le k \le n}|\hat{Z}_k|=\max\big(\max_{1\le k \le n}\hat{Z}_k,\max_{1\le k \le n}(-\hat{Z}_k)\Big)$, we conclude that
    \begin{equation*}
        \lim_{n\to\infty}\frac{\max_{1\le k \le n}|\hat{Z}_k|}{(2\log n)^{1/2}}=
a\quad\text{a.s.}
    \end{equation*}
and thereupon
    \begin{equation*} 
        \lim_{n\to\infty}\frac{ \max_{1\leq k\leq n}\log|\hat{Z}_k|}{\log\log n}= \frac{1}{2}\quad\text{a.s}.
    \end{equation*}
The latter ensures that given $\varepsilon>0$ there exists a sequence $(n_j)_{j\ge1}$ which satisfies
    \begin{equation*}
        \frac{\max_{1\leq k\leq n_j}\log|\hat{Z}_k|}{\log\log n_j}\ge\frac{1}{2}-\varepsilon\quad \text{for all}\quad j\in\mn\quad \text{a.s.}
    \end{equation*}
    Put $\tau_{n_j}=\min\{n\in\mn:\log|\hat{Z}_n|=\max_{1\leq k\leq n_j}\log|\hat{Z}_k|\}$. Then noting that $\tau_{n_j}\leq n_j$ we obtain
    \begin{equation*}
        \frac{\log|\hat{Z}_{\tau_{n_j}}|}{\log\log \tau_{n_j}}=\frac{\max_{1\leq k\leq n_j}\log|\hat{Z}_k|}{\log\log \tau_{n_j}}\ge \frac{\max_{1\leq k\leq n_j}\log|\hat{Z}_k|}{\log\log n_j}\ge\frac{1}{2}-\varepsilon\quad \text{for all}\quad j\in\mn\quad \text{a.s.}
    \end{equation*}
    
    Since $\lim_{j\to\infty}\tau_{n_j}=\infty$ a.s., we infer
    \begin{equation}\label{main2Gsup2}
        \limsup_{n\to\infty}\frac{\log|\hat{Z}_n|}{\log\log n}\ge\frac{1}{2}\quad \text{a.s.}
    \end{equation}
Now \eqref{main2Gsup} follows from \eqref{main2Gsup1} and \eqref{main2Gsup2}.
\end{proof}

We proceed with two corollaries to Lemma \ref{lemma:main2.Gsup}. Put $\hat Y_n:=\eee^{-1/2}\eta_n+\sum_{j=1}^{n-1}2^{j-n}\eta_j$ for $n\in\mn$.
\begin{cor}\label{cor:limit}
Assume that $\eta$ has a standard normal distribution. Then
\begin{equation}\label{main2.Isup}
        \limsup_{n\to\infty}\frac{\log|\hat{Y}_n|}{\log\log n}=\frac{1}{2}\quad \text{{\rm a.s.}}
    \end{equation}
    and
    \begin{equation}\label{main2.Iinf}
        \liminf_{n\to\infty}\frac{\log|\hat{Y}_n|}{\log n}=-1\quad \text{{\rm a.s.}}
    \end{equation}
\end{cor}
\begin{proof}
The variable $\hat Y_n$ admits a representation:
\begin{multline}\label{main2F}
\hat{Y}_n= 
\eee^{-1/2}\eta_n+2^{-1}\eta_{n-1}+\ldots +2^{-(n-1)}\eta_1+2^{-n}\eta_0+2^{-(n+1)}\eta_{-1}+\ldots\\ -2^{-n}\big(\eta_0+2^{-1}\eta_{-1}+2^{-2}\eta_{-2}+\ldots\Big)=:\hat{Z}_n-2^{-n}W,\quad n\in\mathbb{Z},
    \end{multline}
where the series defining $W$ trivially converges a.s. According to \eqref{main2Ginf}, for $\varepsilon\in(0,1)$ and large $n$, $|\hat{Z}_n|\ge n^{-(1+\varepsilon)}$ a.s.
   
This implies that, for large $n$,
    \begin{equation*}
        |\hat{Y}_n|\ge|\hat{Z}_n|-|\hat{Z}_n-\hat{Y}_n|\ge n^{-(1+\varepsilon)}-2^{-n}|W|\quad \text{a.s.}
    \end{equation*}
and thereupon
    \begin{equation}\label{main2.Iinf1}
        \liminf_{n\to\infty}\frac{\log|\hat{Y}_n|}{\log n}\ge-1~~\text{a.s.}
    \end{equation}
    On the other hand, according to \eqref{main2Ginf}, for $\varepsilon\in(0,1)$, $|\hat{Z}_n|\le n^{-(1-\varepsilon)}$ i.o.\ a.s., whence
    
\begin{equation*}
|\hat{Y}_n|\le|\hat{Z}_n|+|\hat{Z}_n-\hat{Y}_n|\le n^{-(1-\varepsilon)}+2^{-n}|W|\quad \text{i.o}\quad \text{a.s.}
    \end{equation*}
This entails 
    \begin{equation}\label{main2.Iinf2}
        \liminf_{n\to\infty}\frac{\log|\hat{Y}_n|}{\log n}\le-1\quad \text{a.s.}
    \end{equation}
The limit relation \eqref{main2.Iinf} follows from \eqref{main2.Iinf1} and \eqref{main2.Iinf2}.

The proof of \eqref{main2.Isup} is based on \eqref{main2Gsup} and proceeds similarly to the proof of \eqref{main2.Iinf}. In view of this, we only provide details for the lower bound. According to \eqref{main2Gsup}, for $\varepsilon\in (0,1/2)$, $|\hat{Z}_n| \ge (\log n)^{1/2-\varepsilon}$ i.o. \ a.s., which entails
\begin{equation*}
|\hat{Y}_n| \ge |\hat{Z}_n| - |\hat{Z}_n - \hat{Y}_n| \ge (\log n)^{1/2-\varepsilon} - 2^{-n}|W| \quad \text{i.o} \quad \text{a.s.}
\end{equation*}
This proves
\begin{equation*}
\limsup_{n\to\infty}\frac{\log|\hat{Y}_n|}{\log\log n} \ge \frac{1}{2} \quad \text{a.s.}
\end{equation*}
\end{proof}
\begin{cor}\label{lemma:main2.D}
Assume that $\eta$ has a standard normal distribution. Then
    \begin{equation}\label{main2.Dsup}
        \limsup_{n\to\infty}\frac{\log(|Y_n|2^{-n})}{\log\log n}=\frac{1}{2}\quad \text{{\rm a.s.}}
    \end{equation}
    and
    \begin{equation}\label{main2.Dinf}
        \liminf_{n\to\infty}\frac{\log(|Y_n|2^{-n})}{\log n}=-1\quad \text{{\rm a.s.}}
    \end{equation}
\end{cor}
\begin{proof}
For $j\in\{1,2,\ldots, n-1\}$, $$  0\leq 1-\exp\big(-\eee^{j^2-n^2}/2\big) \le \eee^{j^2-n^2}/2\le \eee^{-2n+1}/2,$$ which implies that uniformly in $j\in\{1,\ldots, n-1\}$ $$\exp\big(-\eee^{j^2-n^2}/2\big)=1+O(\eee^{-2n}),\quad n\to\infty.$$ As a consequence, recalling that $\hat Y_n=\eee^{-1/2}\eta_n+\sum_{j=1}^{n-1}2^{j-n}\eta_j$ for $n\in\mn$,
\begin{multline}\label{main2.E}
\log(|Y_n|2^{-n})=\log\Big|\eee^{-1/2}\eta_n+\sum_{j=1}^{n-1}2^{j-n}\eta_j\Big|+\log\Big|1+\frac{O\Big(\eee^{-2n}
\big|\sum_{j=1}^{n-1}2^{j-n}\eta_j\big|\Big)}{\eee^{-1/2}\eta_n+\sum_{j=1}^{n-1}2^{j-n}\eta_j}\Big|\\=\log|\hat{Y}_n|+o(1)\quad\text{a.s.}
\end{multline}
The last summand in \eqref{main2.E} is $o(1)$ because
\begin{equation*}
\Big|\frac{\sum_{j=1}^{n-1}2^{j-n}\eta_j}{\eee^{-1/2}\eta_n+\sum_{i=1}^{n-1}2^{i-n}\eta_i}\Big|=\Big|\frac{\sum_{j=1}^{n-1}2^{j-n}\eta_j}{\hat Y_n}\Big|
\end{equation*}
has a subexponential growth, which can be checked as follows. Relation \eqref{main2.Iinf1} implies that, for every $\varepsilon>0$, $|\hat Y_n|\geq n^{-(1+\varepsilon)}$ for large $n$ a.s. Further, $$\Big|\sum_{j=1}^{n-1}2^{j-n}\eta_j\Big|\leq 2^{-1}\sum_{j=1}^{n-1}|\eta_j|=O(n),\quad n\to\infty\quad \text{a.s.}$$ by the strong law of large numbers for standard random walks. Thus, $|\sum_{j=1}^{n-1}2^{j-n}\eta_j/\hat Y_n|=O(n^{2+\varepsilon})$ as $n\to\infty$ a.s. In view of \eqref{main2.E}, relations \eqref{main2.Dsup} and \eqref{main2.Dinf} are equivalent to \eqref{main2.Isup} and \eqref{main2.Iinf}, which hold true by Corollary \ref{cor:limit}.
\end{proof}

Finally, we formulate a corollary to Lemma \ref{lemma:main2.C} and Corollary \ref{lemma:main2.D}.
\begin{cor}\label{lemma:main2.CD}
Assume that $\eta$ has a standard normal distribution. Then
\begin{equation}\label{main2.CD}
    \log\Big|1+\frac{Z_n}{Y_n}\Big|=o(1)\quad n\to\infty\quad \text{{\rm a.s.}}
\end{equation}
\end{cor}
\begin{proof}
Limit relation \eqref{main2.Dinf} entails that, for large $n$, $|Y_n|>n^{-2}2^n$ a.s.

This in combination with Lemma \ref{lemma:main2.C} ensures that
    \begin{equation*}
        \Big|\frac{Z_n}{Y_n}\Big|=o\big(n^2 2^{-n}\exp(-\eee^{2n}/2)\big)=o(1),\quad n\to\infty\quad \text{a.s.}
    \end{equation*}
\end{proof}

\section{Proof of Theorem \ref{thm:main4} 
}\label{sect:proofs}

For notational simplicity we work under the assumption $\sigma^2=1$. This can be achieved by replacing $\eta_k$ with $\eta_k/\sigma$. In what follows, $C$, $C_1,\ldots$ denote positive constants whose values are of no relevance and may change from line to line.

We prove the theorem via a series of lemmas. It is assumed without further notice that the assumptions of Theorem \ref{thm:main4} are in force. However, we note in passing that Condition A is not used in the proofs of Lemmas \ref{lem:prelim} and \ref{4}.
\begin{lemma}\label{lem:prelim}
Let $N$ be a positive integer-valued function satisfying $\lim_{s\to 0+}N(s)=\infty$ and $\lim_{s\to 0+}s^{1/2}\log N(s)=0$. Then $$\lim_{s\to 0+}s^{\beta/2}\sum_{k=1}^{N(s)}\frac{a_k}{k^{1/2+s}}\eta_k=0\quad\text{{\rm a.s.}}$$
\end{lemma}
\begin{proof}
We note in passing that Condition A is not used in the proof that follows.

Put $w_n:=1+\sum_{k=1}^n k^{-1}a_k^2$ for $n\in\mn_0$ (in particular, $w_0=1$). Fix any $\gamma\in (1/2,1)$. Since $$\sum_{k\geq 1}{\rm Var}\Big[\frac{k^{-1/2}a_k\eta_k}{w_k^\gamma}\Big]=\sigma^2\sum_{k\geq 1}\frac{w_k-w_{k-1}}{w_k^{2\gamma}}\leq \sigma^2 \sum_{k\geq 1}\int_{w_{k-1}}^{w_k}x^{-2\gamma}{\rm d}x=\sigma^2 (2\gamma-1)^{-1},$$ the series $\sum_{k\geq 1}w_k^{-\gamma} k^{-1/2}a_k\eta_k$ converges a.s. By Kronecker’s lemma (see, for instance, Theorem 2.5.5 on p.~81 in \cite{Durrett}), $\sum_{k=1}^n k^{-1/2}a_k\eta_k=o(w_n^\gamma)$ a.s. and thereupon $\max_{1\leq k\leq n}|\sum_{j=1}^k j^{-1/2}a_j
\eta_j|=o(w_n^\gamma)$ a.s.

Summation by parts yields $$\sum_{k=1}^{N(s)} k^{-1/2-s}a_k\eta_k=(N(s))^{-s}\sum_{k=1}^{N(s)}k^{-1/2}a_k\eta_k+\sum_{k=1}^{N(s)-1}(k^{-s}-(k+1)^{-s})\sum_{j=1}^k j^{-1/2}a_j\eta_j.$$ In view of  $$(N(s))^{-s}+\sum_{k=1}^{N(s)-1}(k^{-s}-(k+1)^{-s})=1$$ we infer
\begin{multline}\label{eq:inter300}
\Big|\sum_{k=1}^{N(s)} k^{-1/2-s}a_k\eta_k\Big|\leq \max_{1\leq j\leq N(s)}\Big|\sum_{i=1}^j i^{-1/2}a_i\eta_i\Big|=o((w_{N(s)})^\gamma)\\=o((\log N(s))^{\gamma \beta}),\quad s\to 0+\quad\text{a.s.}
\end{multline}
having utilized \eqref{eq:basic} for the last equality. Since $\gamma<1$ the claim follows.
\end{proof}
\begin{remark}
With some extra efforts, we could have proved Lemma \ref{lem:prelim} with  $N$ satisfying \newline $\lim_{s\to 0+}s\log N(s)=0$ rather than $\lim_{s\to 0+}s^{1/2}\log N(s)=0$ and $(s^\beta/(\log\log 1/s))^{1/2}$ instead of $s^{\beta/2}$. However, we refrain from doing so, because at the end Lemma \ref{lem:prelim} will be used with $N(s)=\lfloor 1/s\rfloor$. Thus, the present version of the lemma serves our needs.
\end{remark}

For $k\in\mathcal{S}$, $\rho>0$ and $s\in (0, 1/\eee)$, define the event $$\mathcal{A}_{k,\,\rho}(s):=\Big\{|\eta_k|>\frac{\rho}{|a_k|\log 1/s}\Big(\frac{k^{1+s}}{s^{\beta}\log\log 1/s}\Big)^{1/2}\Big\}.$$ In the sequel, sums like $\sum_{k>N(s)}\ldots$ denote $\sum_{k\in\mathcal{S},\,k>N(s)}\ldots$.
\begin{lemma}\label{lem:inter1}
Let $N$ be a positive integer-valued function satisfying $\lim_{s\to 0+}N(s)=\infty$. For all $\rho>0$,
\begin{equation}\label{eq:event}
\sum_{k>N(s)}\frac{|a_k|}{k^{1/2+s}}|\eta_k|\1_{\mathcal{A}_{k,\,\rho}(s)}=0\quad\text{{\rm for all}}~s>0~\text{{\rm close to}}~0\quad\text{{\rm a.s.}};
\end{equation}
\begin{equation}\label{eq:event2}
\lim_{s\to 0+}s^{\beta/2}\sum_{k>N(s)}\frac{|a_k|}{k^{1/2+s}}\me[|\eta_k|\1_{\mathcal{A}_{k,\,\rho}(s)}]=0.
\end{equation}
\end{lemma}
\begin{proof}
We first show that
\begin{equation}\label{eq:est1}
\frac{\rho}{|a_k|\log 1/s}\Big(\frac{k^{1+s}}{s^\beta\log\log 1/s}\Big)^{1/2}\geq C\frac{k^{1/2}}{|a_k|}\frac{(\log k)^{\beta/2}}{\log\log k (\log^{(3)}k)^{1/2}}
\end{equation}
for all sufficiently large $k\in \mathcal{S}$.

Let $k$ be integers satisfying $\log k\geq 1/s$. The function $x\mapsto x(\log x)^{1/2}$ is increasing for large arguments. Hence, $$\log 1/s (\log\log 1/s)^{1/2}\leq \log\log k (\log^{(3)}k)^{1/2}.$$ Further, the function $s\mapsto k^{s/2}s^{-\beta/2}$ attains its minimum on $(0,\infty)$ at point $s=\beta (\log k)^{-1}$, whence $$k^{s/2}s^{-\beta/2}\geq (\eee/\beta)^{\beta/2} (\log k)^{\beta/2}.$$ Summarizing, \eqref{eq:est1} holds with $c=\rho (\eee/\beta)^{\beta/2}$.

Now let $k\in\mathcal{S}$ satisfy $\log k\leq 1/s$. The function $x\mapsto \eee^{\beta x/2}/(\log x (\log\log x)^{1/2})$ is increasing for large arguments. This secures $$\frac{1}{\log 1/s}\Big(\frac{k^s}{s^\beta\log\log 1/s}\Big)^{1/2}\geq \frac{1}{\log 1/s (s^\beta\log\log 1/s)^{1/2}}\geq \frac{(\log k)^{\beta/2}}{\log\log k (\log^{(3)}k)^{1/2}}.$$ Thus, \eqref{eq:est1} holds with $c=\rho$.

Put $$Q_k:=\frac{k(\log k)^\beta}{a_k^2 (\log\log k)^2\log^{(3)} k},\quad k\in\mathcal{S}.$$ For other $k$ such that $a_k\neq 0$, if any, put $Q_k=1$. This convention applies throughout without further notice.  

\noindent {\sc Proof of \eqref{eq:event}.} Using \eqref{eq:est1} in combination with Condition A we conclude that
$$\sum_{k\in\mathcal{S}}\mmp(\mathcal{A}_{k,\rho}(s))\leq \sum_{k\in\mathcal{S}}\mmp\{\eta^2>CQ_k\}=\me\big[\#\{k\in\mathcal{S}:~Q_k\leq C^{-1}\eta^2\}\big]\leq C_1(1+\me [\eta^2])<\infty.$$ Invoking the direct part of the Borel-Cantelli lemma completes the proof of \eqref{eq:event}.

\noindent {\sc Proof of \eqref{eq:event2}.} By \eqref{eq:est1}, for all sufficiently large $k\in\mathcal S$,
$\mathcal A_{k,\rho}(s)\subseteq \{|\eta_k|>(CQ_k)^{1/2}\}$. Also,
\begin{equation*}
s^{\beta/2}\frac{|a_k|}{k^{1/2+s}}=\frac{(s\log k)^{\beta/2}\exp(-s\log k)}{Q_k^{1/2}\log\log k\,(\log^{(3)}k)^{1/2}}\leq \frac{C_1}{Q_k^{1/2}},
\end{equation*}
because $\sup_{x>0}x^{\beta/2}\eee^{-x}<\infty$. Since $\lim_{s\to 0+}N(s)=\infty$, the finitely many indices for which these estimates may fail do not occur in the sum for all sufficiently
small $s$. Hence
\begin{equation}\label{eq:event2-bound}
s^{\beta/2}\sum_{k>N(s)}\frac{|a_k|}{k^{1/2+s}}\me\big[|\eta_k|\1_{\mathcal A_{k,\rho}(s)}\big]\leq C_1\sum_{\substack{k\in\mathcal S\\k>N(s)}}Q_k^{-1/2}\me\big[|\eta|\1_{\{|\eta|>(CQ_k)^{1/2}\}}\big].
\end{equation}
We intend to show that the series on the right-hand side, with the
restriction $k>N(s)$ removed, converges, which is sufficient for \eqref{eq:event2}. Enumerate
$\mathcal S$ as $\ell_1,\ell_2,\ldots$ so that $Q_{\ell_1}\leq Q_{\ell_2}\leq\ldots$. Such an enumeration exists because Condition~A implies that every bounded subset of $(Q_k)_{k\in\mathcal S}$ is finite.
Moreover, Condition~A gives, for all sufficiently large $j$, $$j\leq\#\{k\in\mathcal S:~Q_k\leq Q_{\ell_j}\}\leq C Q_{\ell_j},$$ and therefore $Q_{\ell_j}\geq C^{-1}j$. Consequently,
\begin{equation*}
\sum_{k\in\mathcal S}Q_k^{-1/2}\me\big[|\eta|\1_{\{|\eta|>(CQ_k)^{1/2}\}}\big]\leq C_2\sum_{j\geq1}j^{-1/2}\me\big[|\eta|\1_{\{|\eta|>C_3 j^{1/2}\}}\big]\leq C_4\me [\eta^2]<\infty.
\end{equation*}

\end{proof}

Put $A(x):=\sum_{1\leq k\leq x}k^{-1}a_k^2$ for $x\geq 1$. Then \eqref{eq:basic} entails $A(\eee^x)\sim cx^\beta$ as $x\to\infty$. With this at hand, we infer
\begin{equation}\label{eq:variance}
v(s):={\rm Var}[X(s)]=\sum_{k\geq 1}k^{-1-2s}a_k^2=\int_{[0,\infty)}\eee^{-2sx}{\rm d}A(\eee^x)~\sim~\frac{c\Gamma(1+\beta)}{(2s)^\beta},\quad s\to 0+
\end{equation}
having utilized Theorem 1.7.1 in \cite{BGT:1989} for the asymptotic equivalence.

For $k\in\mathcal{S}$, $\rho>0$ and $s\in (0,1/\eee)$, the complement of $\mathcal{A}_{k,\,\rho}(s)$ is given by
$$\mathcal{A}^c_{k,\,\rho}(s)=\Big\{|\eta_k|\leq \frac{\rho}{|a_k| (\log 1/s)} \Big(\frac{k^{1+s}}{s^\beta \log\log 1/s}\Big)^{1/2}\Big\}.$$ Put $\tilde \eta_{k,\rho}(s):=\eta_k\1_{\mathcal{A}^c_{k,\,\rho}(s)}-\me [\eta_k\1_{\mathcal{A}^c_{k,\,\rho}(s)}]$ for $k\in\mathcal{S}$ and $s\in (0,1/\eee)$. The random variables so defined are bounded and centered.
\begin{lemma}\label{4}
Let $N$ be a positive integer-valued function. 
Fix any $\gamma\in (0, (\sqrt{5}-1)/2)$, pick any $\rho=\rho(\gamma)$ satisfying
\begin{equation}\label{eq:rho_choice}
(1-\gamma)(1+\gamma)^2(2-\exp(8(1+\gamma)\rho c_\beta))>1,
\end{equation}
where $c_\beta:=(c\,2^{1-\beta}\Gamma(1+\beta))^{-1/2}$, and put $g(s):=(2v(s)\log\log 1/s)^{-1/2}$ for $s\in (0,1/\eee)$ and $s_n:=\exp(-n^{1-\gamma})$ for $n\in\mn$. 
Then
$$\limsup_{n\to\infty}g(s_n)\sum_{k>N(s_n)}\frac{a_k\tilde \eta_{k,\rho}(s_n)}{k^{1/2+s_n}}\leq 1+\gamma\quad \text{{\rm a.s.}}$$
\end{lemma}
\begin{proof}
This proof runs along the lines of the proof of Lemma 5.4 in \cite{Buraczewski etal:2023}. In particular, the existence of so defined $\rho$ is justified in that proof.

Put $$Y(s):=g(s)\sum_{k>N(s)}\frac{a_k\tilde \eta_{k,\rho}(s)}{k^{1/2+s}},\quad s\in (0,1/\eee).$$ Using $\eee^x\leq 1+x+(x^2/2)\eee^{|x|}$ for $x\in\mr$ and $\me[\tilde \eta_{k,\rho}(s)]=0$ we infer, for $u\in\mr$,
\begin{multline*}
\me [\eee^{uY(s)}]=\prod_{k>N(s)}\me \exp\Big(ug(s)\frac{a_k \tilde \eta_{k,\rho}(s)}{k^{1/2+s}}\Big)\\\leq \prod_{k>N(s)} \Big(1+\frac{u^2 (g(s))^2}{2}\frac{a_k^2}{k^{1+2s}}\me\Big[(\tilde \eta_{k,\rho}(s))^2 \exp\Big(|u|g(s)\frac{|a_k\tilde \eta_{k,\rho}(s)|}{k^{1/2+s}}\Big)\Big]\Big).
\end{multline*}
The inequality
\begin{equation}\label{eq:mm1}
\begin{split}
|\tilde \eta_{k,\rho}(s)|  \leq |\eta_k|\1_{\mathcal{A}^c_{k,\,\rho}(s)}+\me (|\eta_k|\1_{\mathcal{A}^c_{k,\,\rho}(s)})&\leq \frac{2\rho k^{(1+s)/2}}{|a_k| (\log 1/s) (s^\beta \log\log 1/s)^{1/2}}\\ &\leq \frac{2\rho k^{1/2+s}}{|a_k| (s^\beta\log\log 1/s)^{1/2}}\quad \text{a.s.}
\end{split}
\end{equation}
holds true for $k\in\mathcal{S}$ and $s\in (0,1/\eee)$. Hence, in view of \eqref{eq:variance},
\begin{equation}\label{eq:mm112}
|\tilde \eta_{k,\rho}(s)| \leq \frac{4\rho c_\beta k^{1/2+s}}{|a_k|}\Big(\frac{2v(s)}{\log\log 1/s}\Big)^{1/2}\quad \text{a.s.}
\end{equation}
for $k\in\mathcal{S}$ and $s>0$ sufficiently close to $0$, which entails $$\exp\Big(|u|g(s)\frac{|a_k\tilde \eta_{k,\rho}(s)|}{k^{1/2+s}}\Big)\leq \exp\Big(\frac{4\rho c_\beta |u|}{\log\log 1/s}\Big)\quad \text{a.s.}$$ This in combination with the inequalities $\me [(\tilde \eta_{k,\rho}(s))^2]\leq 1$ and $ 1+x\leq \eee^x$ for $x\in\mr$ yields, for $u\in\mr$,
\begin{multline}\label{eq:ss1}
\me [\eee^{uY(s)}]\leq \prod_{k>N(s)}\exp\Big(\frac{u^2 (g(s))^2}{2} \frac{a_k^2}{k^{1+2s}}\exp\Big(\frac{4\rho c_\beta |u|}{\log\log 1/s}\Big)\Big)\\ \leq \exp\Big(\frac{u^2}{4\log\log 1/s}\exp\Big(\frac{4\rho c_\beta |u|}{\log\log 1/s}\Big)\Big).
\end{multline}
By Markov's inequality, for $u\geq 0$,
$$\mmp\{Y(s_n)>1+\gamma\}\leq \eee^{-(1+\gamma)u}\me [\eee^{uY(s_n)}]\leq \exp\Big(-(1+\gamma)u+\frac{u^2}{4\log\log 1/s_n}\exp\Big(\frac{4\rho c_\beta u}{\log\log 1/s_n}\Big)\Big).$$ Putting $u=2(1+\gamma)\log\log 1/s_n$ we obtain
\begin{multline*}
\mmp\{Y(s_n)>1+\gamma\}\leq \exp(-(1+\gamma)^2(2-\exp(8(1+\gamma)\rho c_\beta))\log\log 1/s_n)\\=\frac{1}{n^{(1-\gamma)(1+\gamma)^2(2-\exp(8(1+\gamma)\rho c_\beta))}}.
\end{multline*}
Hence, $\sum_{n\geq 1}\mmp\{Y(s_n)>1+\gamma\}<\infty$, and an appeal to the direct part of the Borel-Cantelli lemma completes the proof of Lemma \ref{4}.
\end{proof}

\begin{lemma}\label{5}
Let $N$ be a positive nonincreasing integer-valued function satisfying $\lim_{s\to 0+}N(s)=\infty$ and $\lim_{s\to 0+}s\log N(s)=0$. Let $\rho=\rho(\gamma)$ and $(s_n)_{n\geq 1}$ be as in Lemma \ref{4} with the only difference that $\gamma\in (0,1/2)$. Then $$\lim_{n\to\infty}\sup_{s\in [s_{n+1}, s_n]}\, g(s)\Big|\sum_{k>N(s)}\frac{a_k}{k^{1/2+s}}\tilde \eta_{k,\,\rho}(s)-\sum_{k>N(s_{n+1})} \frac{a_k}{k^{1/2+s_{n+1}}}\tilde \eta_{k,\,\rho}(s_{n+1})\Big|=0\quad\text{{\rm a.s.}}$$
\end{lemma}
\begin{proof}
Let $s\in [s_{n+1}, s_n]$. Using the assumption that $N$ is a nonincreasing function, write
\begin{multline*}
\sum_{k>N(s)}\frac{a_k \tilde \eta_{k,\,\rho}(s)}{k^{1/2+s}}-\sum_{k> N(s_{n+1})}\frac{a_k \tilde \eta_{k,\,\rho}(s_{n+1})}{k^{1/2+s_{n+1}}}\\=
\sum_{k=N(s)+1}^{N(s_{n+1})}\frac{a_k\eta_k}{k^{1/2+s}}-\sum_{k>N(s)} \frac{a_k (\eta_k\1_{\mathcal{A}_{k,\,\rho}(s)}-\me[\eta_k\1_{\mathcal{A}_{k,\,\rho}(s)}])}{k^{1/2+s}}\\+\sum_{k>N(s_{n+1})}\frac{a_k (\eta_k\1_{\mathcal{A}_{k,\,\rho}(s_{n+1})}-\me[\eta_k\1_{\mathcal{A}_{k,\,\rho}(s_{n+1})}])}{k^{1/2+s}}\\+ \sum_{k>N(s_{n+1})}a_k \Big(\frac{1}{k^{1/2+s}}-\frac{1}{k^{1/2+s_{n+1}}}\Big)\tilde \eta_{k,\,\rho}(s_{n+1})=:J_{n,1}(s)-J_{n,2}(s)+J_{n,3}(s)+J_{n,4}(s).
\end{multline*}

\noindent {\sc Analysis of $J_{n,1}(s)$}. By the reasoning leading to \eqref{eq:inter300}, $$\Big|\sum_{k=N(s)+1}^{N(s_{n+1})}\frac{a_k\eta_k}{k^{1/2+s}}\Big|\leq 2\max_{1\leq m\leq N(s_{n+1})}\Big|\sum_{k=1}^m\frac{a_k\eta_k}{k^{1/2}}\Big|.$$ We first prove that
\begin{equation}\label{eq:weightedLIL}
\max_{1\leq m\leq n}\Big|\sum_{k=1}^m\frac{a_k\eta_k}{k^{1/2}}\Big|=O\big((w_n\log\log w_n)^{1/2}\big),\quad n\to\infty \quad\text{a.s.},
\end{equation}
where, as before, $w_n=1+\sum_{k=1}^n k^{-1}a_k^2$ for $n\in\mn_0$.

Let $k_1<k_2<\cdots$ be the increasing enumeration of $\mathcal S$. Put $b_j:=k_j^{-1} a_{k_j}^2$ and $B_j:=\sum_{i=1}^j b_i$ for $j\in\mn$. By \eqref{eq:basic}, $B_j\sim c(\log k_j)^\beta$ as $j\to\infty$.
Recall that $$Q_k=\frac{k(\log k)^\beta}{a_k^2(\log\log k)^2\log^{(3)}k},\quad k\in\mathcal{S}.$$ 
Since, as $j\to\infty$, $$b_j=\frac{(\log k_j)^\beta}{Q_{k_j}(\log\log k_j)^2\log^{(3)}k_j}$$
and $\log\log B_j\sim\log^{(3)}k_j$, we obtain
\begin{equation}\label{eq:effective-level}
\frac{B_j}{b_j\log\log B_j}~\sim~c Q_{k_j}(\log\log k_j)^2.
\end{equation}

We apply Theorem~1.1 in \cite{Tomkins:1983}. First, for each $\varepsilon>0$, \eqref{eq:effective-level} gives
\begin{equation*}
\sum_{j\geq1}\mmp\Big\{b_j^{1/2}|\eta_{k_j}|>\varepsilon \Big(\frac{B_j}{\log\log B_j}\Big)^{1/2}\Big\}\leq \sum_{k\in\mathcal S}\mmp\{\eta^2>C_\varepsilon Q_k\}<\infty
\end{equation*}
and thereupon $$\lim_{j\to\infty}\Big(\frac{b_j\log\log B_j}{B_j}\Big)^{1/2}\eta_{k_j}=0\quad\text{a.s.}$$ by the direct part of the Borel-Cantelli lemma. To justify the convergence of the last series, observe that  Condition~A implies
$$\#\{k\in\mathcal S:Q_k\leq x\}\leq C(1+x), \quad x>0,$$ and therefore $$\sum_{k\in\mathcal{S}}\mmp\{\eta^2>C_\varepsilon Q_k\}=\me\big[\#\{k\in\mathcal S:Q_k<C_\varepsilon^{-1}\eta^2\}\big]\leq C\big(1+\me[\eta^2]\big)<\infty.$$

Moreover, \eqref{eq:effective-level} and $\lim_{j\to\infty}Q_{k_j}=\infty$ imply that  
$$\lim_{j\to\infty}\Big(\frac{B_j}{b_j\log\log B_j}\Big)^{1/2}=\infty.$$ Consequently, for each $\varepsilon>0$, 
$\lim_{j\to\infty} \me\big[\eta^2\1_{\{|\eta|>\varepsilon (B_j/(b_j\log\log B_j))^{1/2}\}}\big]=0$. Since $B_n=\sum_{j\leq n}b_j\to\infty$ as $n\to\infty$, the Toeplitz lemma yields
$$\lim_{n\to\infty}\frac{1}{B_n}\sum_{j=1}^n b_j\me\big[\eta^2\1_{\{|\eta|>\varepsilon (B_j/(b_j\log\log B_j))^{1/2}\}}\big]=0.$$ Hence, the modified Lindeberg functions appearing in
Theorem~1.1 of \cite{Tomkins:1983} converge to $1$. By that theorem, $$\limsup_{n\to\infty}\frac{\sum_{j=1}^n k_j^{-1/2}a_{k_j}\eta_{k_j}}{(2B_n\log\log B_n)^{1/2}}=1 \quad\text{a.s.}$$ with the corresponding lower limit equal to $-1$. Since $(B_j)_{j\geq 1}$ is a nondecreasing sequence and the function $x\mapsto (x\log\log x)^{1/2}$ is eventually increasing, we conclude 
that $$\max_{1\leq j\leq n}\Big|\sum_{i=1}^j k_i^{-1/2}a_{k_i}\eta_{k_i}\Big|=O((B_n\log\log B_n)^{1/2}),\quad n\to\infty\quad\text{a.s.}$$ Recall the notation $\rho(x)=\#\{k\in\mathcal{S}:~k\leq x\}$. Noting that $$B_{\rho(n)}=\sum_{k_j\leq n}k_j^{-1}a_{k_j}^2=\sum_{k=1}^n k^{-1}a_k^2=w_n-1$$ we finally obtain
\begin{multline*} \max_{1\leq m\leq n}\Big|\sum_{k=1}^m k^{-1/2}a_k\eta_k \Big|=
\max_{1\leq j\leq \rho(n)}\Big|\sum_{i=1}^j k_i^{-1/2}a_{k_i}\eta_{k_i}\Big|=O\big((B_{\rho(n)}\log\log B_{\rho(n)})^{1/2}\big)\\=O\big((w_n\log\log w_n)^{1/2}\big),\quad n\to\infty\quad\text{a.s.}
\end{multline*}

Now, uniformly over $s\in [s_{n+1}, s_n]$,
\begin{multline*}
g(s)|J_{n,1}(s)|\leq 2g(s_n)\max_{1\leq m\leq N(s_{n+1})} \Big|\sum_{k=1}^m\frac{a_k\eta_k}{k^{1/2}}\Big|\\=O\big(g(s_n)(w_{N(s_{n+1})}\log\log w_{N(s_{n+1})})^{1/2}\big),\quad n\to\infty\quad \text{a.s.}
\end{multline*}
By \eqref{eq:basic}, $w_n\sim c(\log n)^\beta$ and, by \eqref{eq:variance}, $g(s_n)\sim g(s_{n+1})\sim {\rm const}\,s_{n+1}^{\beta/2}(\log\log 1/s_{n+1})^{-1/2}$. Hence, the last expression is $$O\Big((s_{n+1}\log N(s_{n+1}))^{\beta/2}\Big(\frac{\log^{(3)} N(s_{n+1})}{\log\log 1/s_{n+1}}\Big)^{1/2}\Big)=o(1),\quad n\to\infty\quad\text{a.s.}$$ Indeed, $\lim_{n\to\infty}s_{n+1}\log N(s_{n+1})=0$  by the choice of $N(s)$. The latter limit relation also implies that the square-root factor is bounded. Thus,
$$\lim_{n\to\infty}\sup_{s\in[s_{n+1},s_n]}g(s)|J_{n,1}(s)|=0\quad\text{a.s.}$$

\noindent {\sc Analysis of $J_{n,2}(s)$ and $J_{n,3}(s)$}. We only treat $J_{n,3}(s)$. The argument for $J_{n,2}(s)$ is analogous. For $s\in [s_{n+1}, s_n]$,
\begin{multline*}
g(s)|J_{n,3}(s)|\\\leq\frac{g(s_n)}{g(s_{n+1})} g(s_{n+1})\sum_{k> N(s_{n+1})} \frac{|a_k|(|\eta_k|\1_{\mathcal{A}_{k,\,\rho}(s_{n+1})}+\me [|\eta_k|\1_{\mathcal{A}_{k,\,\rho}(s_{n+1})}])}{k^{1/2+s}}~\to~ 0,\quad n\to\infty\quad \text{a.s.},
\end{multline*}
where the limit relation is secured by Lemma \ref{lem:inter1} and $\lim_{n\to\infty}(g(s_n)/g(s_{n+1}))=1$. Here, we have used $g(s)\leq g(s_n)$ and $k^{-1/2-s}\leq k^{-1/2-s_{n+1}}$.
Thus, $$\lim_{n\to\infty}\sup_{s\in [s_{n+1},s_n]}g(s)|J_{n,3}(s)|=0\quad\text{a.s.}$$  

\noindent {\sc Analysis of $J_{n,4}(s)$}. For $n\in\mn$ and $u>0$, put
$$Y_n(u):=\sum_{k>N(s_{n+1})}\frac{a_k\tilde\eta_{k,\rho}(s_{n+1})}{k^{1/2+u}}.$$ Since $J_{n,4}(s)=Y_n(s)-Y_n(s_{n+1})$,  it is enough to prove that
\begin{equation}\label{eq:3}
\lim_{n\to\infty}
\sup_{u\in[s_{n+1},s_n]}
g(s_n)|Y_n(u)-Y_n(s_{n+1})|=0
\quad\text{{\rm a.s.}}
\end{equation}
Indeed, $g(s)\leq g(s_n)$ for $s\in[s_{n+1},s_n]$ and all
sufficiently large $n$.

We first note that, almost surely, the same series with $u$ replaced by a complex variable defines an analytic extension of $Y_n$ to $H_0:=\{z\in\mathbb{C}: {\rm Re}\, z>0\}$. Indeed, for every $q\in\mathbb Q\cap(0,\infty)$, where $\mathbb{Q}$ is the set of rational numbers,
$$\sum_{k>N(s_{n+1})}{\rm Var}\Big[\frac{a_k\widetilde\eta_{k,\,\rho}(s_{n+1})}{k^{1/2+q}}\Big]\leq
\sum_{k\geq1}\frac{a_k^2}{k^{1+2q}}<\infty,$$
where we have used ${\rm Var}[\widetilde\eta_{k,\,\rho}(s_{n+1})]\leq \me[\eta_k^2]=1$. Since the summands are independent and centered, the corresponding
series converges almost surely. Taking the intersection over
$q\in\mathbb Q\cap(0,\infty)$, we conclude that, for every
$n\in\mn$, its abscissa of convergence is a.s.\ at most zero.
Hence, by the general theory of Dirichlet series (for instance, Sections~9.11-9.12 in \cite{Titchmarsh:1939}), the same series defines
an analytic extension of $Y_n$ to $H_0$. Since $\mn$ is countable,
this holds simultaneously for all $n$ on an event of probability
one. In particular, on this event every $Y_n$ is continuous on
$(0,\infty)$.

For $j\in\mn_0$ and $n\in\mn$, put $$F_j(n):=\big\{t_{j,m}:=s_{n+1}+2^{-j}m(s_n-s_{n+1}):\ 0\leq m\leq 2^j\big\}$$. Then $F_j(n)\subseteq F_{j+1}(n)$ and
$\bigcup_{j\geq0}F_j(n)$ is dense in $[s_{n+1},s_n]$.
For $u\in[s_{n+1},s_n]$, define $$u_j:=\max\{v\in F_j(n):v\leq u\}
=s_{n+1}+2^{-j}(s_n-s_{n+1})\Big\lfloor\frac{2^j(u-s_{n+1})}{s_n-s_{n+1}}
\Big\rfloor.$$ Then $\lim_{j\to\infty} u_j=u$. Moreover, for every $j\in\mn$, either $u_j=u_{j-1}$ or $u_j-u_{j-1}=2^{-j}(s_n-s_{n+1})$. Consequently, if $u_j=t_{j,m}$, then either $u_{j-1}=t_{j,m}$
or $u_{j-1}=t_{j,m-1}$.

We now use the almost-sure continuity established above: since $\lim_{\ell\to\infty}u_\ell=u$,
$$\lim_{\ell\to\infty} Y_n(u_\ell)=Y_n(u) \quad\text{a.s.}$$
Therefore,
\begin{multline*}
|Y_n(u)-Y_n(s_{n+1})|=\lim_{\ell\to\infty} |Y_n(u_\ell)-Y_n(s_{n+1})|\\\leq \lim_{\ell\to\infty} \sum_{j=0}^{\ell}\max_{1\leq m\leq 2^j}|Y_n(t_{j,m})-Y_n(t_{j,m-1})|
=\sum_{j\geq0} \max_{1\leq m\leq 2^j}|Y_n(t_{j,m})-Y_n(t_{j,m-1})|.
\end{multline*}
Since the right-hand side does not depend on $u$, it follows that
$$\sup_{u\in[s_{n+1},s_n]}|Y_n(u)-Y_n(s_{n+1})|\leq \sum_{j\geq 0} \max_{1\leq m\leq 2^j}|Y_n(t_{j,m})-Y_n(t_{j,m-1})|.$$ Since $\sum_{j\geq0} (j+1)2^{-j}=4$, it suffices to prove that, for all $\varepsilon>0$,
\begin{equation}\label{eq:5}
\sum_{n\geq1}\sum_{j\geq0} \mmp\Big\{\max_{1\leq m\leq 2^j} g(s_n)|Y_n(t_{j,m})-Y_n(t_{j,m-1})|>\varepsilon (j+1) 2^{-j}\Big\}<\infty.
\end{equation}

Applying the mean value theorem separately to each 
function $x\mapsto k^{-1/2-x}$, we obtain deterministic numbers $r_{j,m,k}\in [t_{j,m-1}, t_{j,m}]$ such that $$k^{-1/2-t_{j,m}}-k^{-1/2-t_{j,m-1}}= -(t_{j,m}-t_{j,m-1})\frac{\log k}{k^{1/2+r_{j,m,k}}}.$$ Consequently,
\begin{equation}\label{eq:proof_mean_value_thm}
Y_n(t_{j,m})-Y_n(t_{j,m-1})\\
=
-2^{-j}(s_n-s_{n+1})
\sum_{k>N(s_{n+1})}
\frac{a_k\log k}{k^{1/2+r_{j,m,k}}}
\tilde\eta_{k,\rho}(s_{n+1}).
\end{equation}

We now follow the exponential-moment argument used in the proof of
Lemma~\ref{4}. Since
$r_{j,m,k}\geq t_{j,m-1}\geq s_{n+1}$, we have $$\frac{1}{k^{1+2r_{j,m,k}}}\leq \frac{1}{k^{1+2s_{n+1}}}\quad\text{and}\quad \frac{\log k}{k^{1/2+r_{j,m,k}}} \leq \frac{\log k}{k^{1/2+s_{n+1}}}.$$
Moreover, for $k\geq2$, $$\frac{\log k}{k^{s_{n+1}/2}}\leq\frac{2}{\eee s_{n+1}}\leq\frac{1}{s_{n+1}}.$$ Consequently, the definition of the truncation gives
\begin{multline*}
|\widetilde\eta_{k,\rho}(s_{n+1})|
\leq
\frac{2\rho k^{(1+s_{n+1})/2}}
{|a_k|(\log 1/s_{n+1})
 (s_{n+1}^{\beta}\log\log 1/s_{n+1})^{1/2}}\\
\leq
\frac{2\rho k^{1/2+s_{n+1}}}
{|a_k|\log k\,(\log 1/s_{n+1})
 (s_{n+1}^{2+\beta}\log\log 1/s_{n+1})^{1/2}}
\quad\text{a.s.}
\end{multline*}
Terms corresponding to $a_k=0$ may, of course, be omitted.

Using $\me[\widetilde\eta_{k,\rho}(s_{n+1})]=0$,
$\me[(\widetilde\eta_{k,\rho}(s_{n+1}))^2]\leq1$ and
$\eee^x\leq1+x+(x^2/2)\eee^{|x|}$, we obtain, for $u\in\mr$,
\begin{multline}
\me\exp\Big[\pm u\sum_{k>N(s_{n+1})}\frac{a_k\log k}{k^{1/2+r_{j,m,k}}}\widetilde\eta_{k,\rho}(s_{n+1})\Big]\\ \leq \exp\Big(\frac{u^2}{2}\exp\Big(\frac{2\rho |u|}{(\log 1/s_{n+1})(s_{n+1}^{2+\beta}\log\log 1/s_{n+1})^{1/2}}\Big)\sum_{k>N(s_{n+1})}\frac{(a_k\log k)^2}{k^{1+2s_{n+1}}}\Big).\label{eq:proof_mgf_preestimate}
\end{multline}
Since $v$ is completely monotone, two applications of the monotone
density theorem to \eqref{eq:variance} (Theorem 1.7.2 and the remark following it in \cite{BGT:1989}) yield
$$\sum_{k\geq1}\frac{(a_k\log k)^2}{k^{1+2s}}=\frac{v''(s)}{4}~\sim~\frac{c\beta\Gamma(\beta+2)}{(2s)^{\beta+2}},
\quad s\to0+.$$
Thus, with $b_\beta:=c\beta\Gamma(\beta+2)2^{-1-\beta}$, the sum in \eqref{eq:proof_mgf_preestimate} is at most
$b_\beta s_{n+1}^{-2-\beta}$ for all sufficiently large $n$.
Therefore,
\begin{multline}
\me\exp\Big[\pm u\sum_{k>N(s_{n+1})}\frac{a_k\log k}{k^{1/2+r_{j,m,k}}}\widetilde\eta_{k,\rho}(s_{n+1})\Big]\\ \leq
\exp\Big(\frac{b_\beta u^2}{2s_{n+1}^{2+\beta}}\exp\Big(\frac{2\rho |u|}{(\log 1/s_{n+1})(s_{n+1}^{2+\beta}\log\log 1/s_{n+1})^{1/2}}\Big)\Big).
\label{eq:proof_mgf_estimate}
\end{multline}
Let $u>0$. Combining \eqref{eq:proof_mean_value_thm},
\eqref{eq:proof_mgf_estimate}, Markov's inequality and
$\eee^{u|x|}\leq\eee^{ux}+\eee^{-ux}$, we obtain
\begin{multline*}
\mmp\big\{g(s_n)|Y_n(t_{j,m})-Y_n(t_{j,m-1})|>\varepsilon (j+1)2^{-j}\big\}\\
\leq 2\exp\Big(-\frac{u\varepsilon(j+1)}{g(s_n)(s_n-s_{n+1})}+\frac{b_\beta u^2}{2s_{n+1}^{2+\beta}}\exp\Big(\frac{2\rho u}{(\log 1/s_{n+1})(s_{n+1}^{2+\beta}\log\log 1/s_{n+1})^{1/2}}\Big)\Big).
\end{multline*}

Put $$u=\frac{3\varepsilon s_{n+1}^{2+\beta}}{2b_\beta g(s_n)(s_n-s_{n+1})}, \qquad k_n:=\frac{s_{n+1}^{2+\beta}}{b_\beta(g(s_n))^2(s_n-s_{n+1})^2},$$
and $$\ell_n:=\frac{s_{n+1}^{1+\beta/2}}{b_\beta g(s_n)(s_n-s_{n+1})(\log 1/s_{n+1})(\log\log 1/s_{n+1})^{1/2}}.$$ Then $$\ell_n~\sim~\frac{1}{b_\beta(1-\gamma)c_\beta}n^{2\gamma-1}~\to~ 0$$ because $\gamma<1/2$. Hence
$\exp(3\rho\varepsilon\ell_n)\leq10/9$ for all sufficiently large
$n$, and therefore
\begin{equation*}
\mmp\big\{g(s_n)|Y_n(t_{j,m})-Y_n(t_{j,m-1})|>\varepsilon (j+1)2^{-j}\big\}\leq 2\exp\Big(-\frac{3}{2}\varepsilon^2(j+1)k_n
+\frac{5}{4}\varepsilon^2k_n\Big).
\end{equation*}
Furthermore,
\begin{equation}\label{eq:4}
k_n~\sim~\frac{1}{b_\beta(1-\gamma)c_\beta^2}n^{2\gamma}\log n~\to~\infty.
\end{equation}
Using the union bound, for all sufficiently large $n$,
\begin{multline*}
\sum_{j\geq0}\mmp\big\{\max_{1\leq m\leq2^j}g(s_n)|Y_n(t_{j,m})-Y_n(t_{j,m-1})|
>\varepsilon (j+1)2^{-j}\big\}\\ \leq \sum_{j\geq0}2^{j+1}
\exp\Big(-\frac{3}{2}\varepsilon^2(j+1)k_n
+\frac{5}{4}\varepsilon^2k_n\Big)
=\frac{2\exp(-\varepsilon^2k_n/4)}
 {1-2\exp(-3\varepsilon^2k_n/2)}.
\end{multline*}
By \eqref{eq:4}, the last expression is summable in $n$. This proves \eqref{eq:5}, and the Borel-Cantelli lemma completes the proof of \eqref{eq:3}.
\end{proof}

\begin{lemma}\label{lem:l1}
Fix any $\gamma>0$ and put $\ms_n:=\exp(-n^{1+\gamma})$ for integer $n\geq 2$. Let $N_1$ and $N_2$ be integer-valued, possibly dependent on $\gamma$, functions defined on some (not necessarily the same) right vicinities of $0$ and satisfying $\lim_{s\to 0+}(N_1(s))^s=1$ and $\lim_{s\to 0+}(N_2(s))^s=+\infty$. Then
\begin{equation}\label{eq:inter}
\lim_{n\to\infty} g(\ms_n)\sum_{k=1}^{N_1(\ms_n)}\frac {a_k}{k^{1/2+\ms_n}} \tilde{\eta}_{k,1}(\ms_n)=0\quad\text{{\rm a.s.}}
\end{equation}
and
\begin{equation}\label{eq:inter1}
\lim_{n\to\infty} g(\ms_n)\sum_{k>N_2(\ms_n)}\frac {a_k}{k^{1/2+\ms_n}} \tilde{\eta}_{k,1}(\ms_n)=0\quad\text{{\rm a.s.}}
\end{equation}
\end{lemma}
\begin{proof}
For $s>0$ close to $0$, put $$Z_1(s):=g(s)\sum_{k=1}^{N_1(s)}\frac {a_k }{k^{1/2+s}} \tilde{\eta}_{k,1}(s).$$ The argument leading to both \eqref{eq:inter} and \eqref{eq:inter1} is similar to that used in the proof of Lemma \ref{4}. In view of this, we provide a proof of \eqref{eq:inter} and only comment on a proof of \eqref{eq:inter1}.

As far as \eqref{eq:inter} is concerned, according to the direct part of the Borel-Cantelli lemma, it is sufficient to prove that, for all $\varepsilon > 0$,
\begin{equation}\label{eq:w1}
\sum_{n\ge 1} \mmp\{|Z_1(\ms_n)|>\varepsilon\} <\infty.
\end{equation}
To this end, we obtain (compare with \eqref{eq:ss1}), for $u\in\mr$,
$$
\me [\eee^{u Z_1(s)}] \le   \exp \Big( \frac{u^2 (g(s))^2}{2}\sum_{k=1}^{N_1(s)}\frac {a_k^2}{k^{1+2s}}\exp\Big( \frac{4c_\beta |u|}{\log\log 1/s} \Big) \Big).$$
We claim that
\begin{equation}\label{eq:small}
\sum_{k=1}^{N_1(s)}\frac{a_k^2}{k^{1+2s}}=o(v(s))\quad\text{and}\quad \sum_{k>N_2(s)}\frac{a_k^2}{k^{1+2s}}=o(v(s)) ,\quad s\to 0+.
\end{equation}
Indeed, using \eqref{eq:basic} we conclude that $$s^\beta \sum_{k=1}^{N_1(s)}\frac{a_k^2}{k^{1+2s}}\leq s^\beta\sum_{k=1}^{N_1(s)}\frac{a_k^2}{k} ~\sim~ c (s\log N_1(s))^\beta~\to~ 0,\quad s\to 0+.$$ Invoking \eqref{eq:variance} we obtain the first part of \eqref{eq:small}. Further, observe that $\lim_{s\to 0+}(N_2(s))^s=+\infty$ entails $\lim_{s\to 0+}N_2(s)=+\infty$. Recall that $$A(x)=\sum_{1\leq k\leq x}k^{-1}a_k^2~\sim~c(\log x)^\beta,\quad x\to\infty.$$ Hence, given $\varepsilon>0$, $A(x)\leq (c+\varepsilon)(\log x)^\beta$ for $x\geq N_2(s)$ and $s>0$ close to $0$. Using this in combination with integration by parts we obtain
\begin{multline*}
\sum_{k>N_2(s)}\frac{a_k^2}{k^{1+2s}}=\int_{(N_2(s),\infty)}x^{-2s}{\rm d}A(x)=-(N_2(s))^{-2s}A(N_2(s))+2s \int_{N_2(s)}^\infty x^{-2s-1}A(x){\rm d}x\\\leq 2(c+\varepsilon)s \int_{N_2(s)}^\infty x^{-2s-1}(\log x)^\beta{\rm d}x.
\end{multline*}
Changing the variable $y=2s\log x$ yields $$s^\beta \sum_{k>N_2(s)}\frac{a_k^2}{k^{1+2s}}\leq 2^{-\beta}(c+\varepsilon)\int_{2s\log N_2(s)}^\infty \eee^{-y}y^\beta{\rm d}y~\to~0,\quad s\to 0+,$$ and the second part of \eqref{eq:small} follows.

Pick $r>0$ close to $0$ to ensure that, with $\varepsilon$ as in \eqref{eq:w1}, $\delta:=r\varepsilon^{-2}\exp(4\varepsilon^{-1}c_\beta)$ satisfies $(1-\delta)(1+\gamma)>1$. According to the preceding discussion, for small enough $s>0$, $$\sum_{k=1}^{N_1(s)} \frac{a_k^2}{k^{1+2s}}\leq rv(s).$$ Summarizing, for $u\in\mr$ and small $s>0$,
\begin{multline*}
\me [\eee^{u Z_1(s)}] \le   \exp \Big( \frac{u^2 (g(s))^2}{2}rv(s)\exp\Big( \frac{4c_\beta |u|}{\log\log 1/s} \Big) \Big)\\
=\exp \Big(\frac{r u^2}{4 \log\log 1/s} \exp\Big( \frac{4 c_\beta |u|}{\log\log 1/s} \Big) \Big).
\end{multline*}
Invoking Markov's inequality with $u = (1/\varepsilon)\log\log(1/s)$ yields, for small $s>0$,
\begin{align*}
\mmp\{Z_1 (s)>\varepsilon\} \le \eee^{-u\varepsilon} \me [\eee^{uZ_1(s)}] \le \exp\big(-\big(1 - r\varepsilon^{-2} \eee^{4 \varepsilon^{-1} c_\beta}\big) \log\log(1/s)    \big)
= \frac 1{\log(1/s)^{1-\delta}}.
\end{align*}
Noting that the same exponential-moment estimate applies to $-Z_1(s)$, we arrive at \eqref{eq:w1}.

The proof of \eqref{eq:inter1} uses the second relation in \eqref{eq:small} and is otherwise entirely analogous to the proof of \eqref{eq:inter}. We omit the details.
\end{proof}
\begin{lemma}\label{lem:new}
Fix sufficiently small $\delta >0$, pick $\gamma >0$ satisfying $(1+\gamma)(1-\delta^2/8)<1$ and put $\ms_n=\exp(-n^{1+\gamma})$ for integer $n\geq 2$. Then
$$
{\lim\sup}_{n\to \infty}g(\ms_n) \sum_{k\geq 1}\frac{a_k}{k^{1/2+\ms_n}}\tilde{\eta}_{k,1}(\ms_n) \geq 1-\delta \quad\text{{\rm a.s.}}
$$
\end{lemma}
\begin{proof}
We follow the proof of Lemma~5.8 in
\cite{Buraczewski etal:2023}. Choose the integer-valued functions
$N_1$ and $N_2$ used there. Thus, $\lim_{s\to 0+}(N_1(s))^s=1$, $\lim_{s\to 0+}(N_2(s))^s=+\infty$ and, for some $n_0\in\mn$,
\begin{equation}\label{choice}
N_1(\ms_{n+1})\geq N_2(\ms_n),\qquad n\geq n_0.
\end{equation}
For sufficiently small $s>0$, put
$$Z_2(s):=g(s)\sum_{k=N_1(s)+1}^{N_2(s)}\frac{a_k}{k^{1/2+s}}\widetilde\eta_{k,1}(s).$$
By Lemma~\ref{lem:l1}, it suffices to prove
\begin{equation}\label{eq:ss9}
\limsup_{n\to\infty}Z_2(\ms_n)\geq1-\delta
\quad\text{a.s.}
\end{equation}

To this end, the proof of Lemma~5.8 in \cite{Buraczewski etal:2023} applies once
the analogues of its variance and fourth-order estimates are
established. First, \eqref{eq:small} and \eqref{eq:variance} give
\begin{equation}\label{eq:blockvariance}
\lim_{s\to 0+} s^\beta\sum_{k=N_1(s)+1}^{N_2(s)}
\frac{a_k^2}{k^{1+2s}}=\frac{c\Gamma(1+\beta)}{2^\beta}
=\frac{1}{2c_\beta^2}.
\end{equation}
The truncation levels tend to $\infty$ uniformly over
$k\in\mathcal S\cap(N_1(s),N_2(s)]$. Indeed, formula \eqref{eq:est1}
shows that the square of the truncation level is bounded below by $C Q_k$. Condition~A implies that $Q_k\to\infty$ as
$k\to\infty$ through $\mathcal{S}$. Since $\lim_{s\to 0+}N_1(s)=\infty$, it follows that
$$
\inf_{\substack{k\in\mathcal S\\N_1(s)<k\leq N_2(s)}}Q_k
\geq
\inf_{\substack{k\in\mathcal S\\k>N_1(s)}}Q_k~\to~ \infty,
\qquad s\to0+.
$$
Thus the required uniform divergence follows. Consequently, $\lim_{s\to 0+}\me[\widetilde\eta_{k,1}^2(s)]=1$ uniformly over the same range of $k$. Hence, on putting
$$\xi_k(s):=
\frac{s^{\beta/2}a_k}{k^{1/2+s}}
\widetilde\eta_{k,1}(s),$$
we obtain from \eqref{eq:blockvariance}
\begin{equation}\label{eq:ss4}
\lim_{s\to 0+} \sum_{k=N_1(s)+1}^{N_2(s)}\me[\xi_k^2(s)]=\frac{1}{2c_\beta^2}.
\end{equation}

It remains to verify the fourth-order estimate, namely,
\begin{equation}\label{eq:fourth-order}
u^2s^{2\beta}\sum_{k\geq1} \frac{a_k^4}{k^{2+4s}}=o(1), \quad s\to0+
\end{equation}
with $u=O\big((\log\log(1/s))^{1/2}\big)$.
This estimate follows from \eqref{eq:variance} and Condition~A. Indeed, according to \eqref{eq:variance}, $\sum_{k\geq1}k^{-1-2s}a_k^2=O(s^{-\beta})$ as $s\to0+$. Consequently,
\begin{equation*}
u^2s^{2\beta}\sum_{k\geq1}\frac{a_k^4}{k^{2+4s}}\leq u^2s^{2\beta}\sup_{k\geq1}\frac{a_k^2}{k^{1+2s}}\sum_{k\geq1}\frac{a_k^2}{k^{1+2s}}=O\Big(u^2s^\beta\sup_{k\geq1}\frac{a_k^2}{k^{1+2s}}\Big).
\end{equation*}
Thus, we are left with proving that
\begin{equation}\label{eq:maximal-summand}
\lim_{s\to 0+}s^\beta\log\log(1/s)\,\sup_{k\geq1}\frac{a_k^2}{k^{1+2s}}=0.
\end{equation}

Recall the notation $$Q_k:=\frac{k(\log k)^\beta}{a_k^2(\log\log k)^2\log^{(3)}k},\quad k\in\mathcal{S}.$$ Condition~A implies that there exists $K\in\mn$ such that $Q_k\geq 1$ for all $k\in\mathcal{S}$ satisfying $k\geq K$.
Furthermore,
\begin{equation}\label{eq:single-term}
s^\beta\frac{a_k^2}{k^{1+2s}}=\frac{(s\log k)^\beta k^{-2s}}{Q_k(\log\log k)^2\log^{(3)}k}.
\end{equation}
For $k\geq K$, we split the supremum in \eqref{eq:maximal-summand} according as $\log k\leq s^{-1/2}$ or $\log k>s^{-1/2}$. In the first range, $(s\log k)^\beta k^{-2s}\leq s^{\beta/2}$, and hence its contribution to the left-hand side of
\eqref{eq:maximal-summand} vanishes. In the second range, $\log\log k\geq 2^{-1}\log(1/s)$ for all sufficiently small $s$, whereas $\sup_{x>0}x^\beta\eee^{-2x}<\infty$.
It follows from \eqref{eq:single-term} that the contribution of this range is $O((\log(1/s))^{-2})=o(1)$. For the
remaining indices $k<K$, $$s^\beta \log\log(1/s)\sup_{1\leq k<K}\frac{a_k^2}{k^{1+2s}}\leq s^\beta\log\log(1/s)\max_{1\leq k<K}\frac{a_k^2}{k}~\to~0,\quad s\to 0+.$$ This proves \eqref{eq:maximal-summand}, and hence \eqref{eq:fourth-order}.

Together with \eqref{eq:ss4} and the truncation bound $\max_k|u\xi_k(s)|=o(1)$, which is secured by $$|u\xi_k(s)|\leq O\big((\log\log 1/s)^{1/2}\big)\frac{s^{\beta/2}|a_k|}{k^{1/2+s}}\frac{2}{|a_k|\log 1/s}\Big(\frac{k^{1+s}}{s^\beta\log\log 1/s}\Big)^{1/2}=\frac{O(1)}{\log 1/s}=o(1)$$ as $s\to 0+$,
 
formula \eqref{eq:fourth-order} gives, exactly as in
(5.28)--(5.32) of \cite{Buraczewski etal:2023},
\begin{equation}\label{eq:p1}
\me[\exp(us^{\beta/2}W(s))]=\exp\Big(\frac{u^2}{4c_\beta^2}+u^2r(s,u)\Big),\quad s\to 0+
\end{equation}
for an appropriate function $r$ satisfying, for each fixed $C>0$, $\lim_{s\to 0+}\sup_{|u|\leq C(\log\log 1/s)^{1/2}}|r(s,u)| = 0$,
where $W(s):=Z_2(s)/g(s)$. The remainder of the proof is the exponential-change-of-measure
argument from Lemma~5.8 in \cite{Buraczewski etal:2023}. Namely,
using \eqref{eq:p1},
$$\frac{s^{\beta/2}}{g(s)}~\sim~ c_\beta^{-1}(\log\log(1/s))^{1/2},\quad s\to 0+$$
and choosing $u(s)=2(1-\delta/2)c_\beta (\log\log(1/s))^{1/2}$, one obtains, for all sufficiently small $s>0$,
\begin{equation}\label{eq:r3}
\mmp\{Z_2(s)>1-\delta\}
\geq
\frac{1}{3}
\exp\Big(-\Big(1-\frac{\delta^2}{8}\Big)\log\log(1/s)\Big).
\end{equation}
Therefore, $$\sum_{n\geq n_0}\mmp\{Z_2(\ms_n)>1-\delta\}=\infty,$$ because $(1+\gamma)(1-\delta^2/8)<1$. By \eqref{choice}, the random variables $Z_2(\ms_{n_0})$, $Z_2(\ms_{n_0+1}),\ldots$ are independent. The converse part of the
Borel--Cantelli lemma now yields \eqref{eq:ss9}, completing the proof of Lemma \ref{lem:new}.
\end{proof}

\noindent {\sc Completion of the proof of Theorem~\ref{thm:main4}}. Put $N(s):=\lfloor 1/s\rfloor$ for $0<s<1/\eee$. Then $N$ is nonincreasing, $$N(s)\to\infty, \qquad s^{1/2}\log N(s)\to 0\qquad\text{and}\qquad (N(s))^s\to 1
$$ as $s\to 0+$. Recall that $g(s):=\big(2v(s)\log\log(1/s)\big)^{-1/2}$.

An application of Lemma~\ref{lem:prelim}, with this choice of $N$, gives
$$\lim_{s\to 0+} g(s)\sum_{k\leq N(s)}\frac{a_k\eta_k}{k^{1/2+s}}=0\quad\text{a.s.}$$ Moreover, Lemma~\ref{lem:inter1} shows that, for each $\rho>0$,
$$\lim_{s\to 0+} g(s)\sum_{k>N(s)}\frac{a_k}{k^{1/2+s}}\big(\eta_k\1_{\mathcal A_{k,\rho}(s)}-\me[\eta_k\1_{\mathcal A_{k,\rho}(s)}]\big)=0\quad\text{a.s.}$$
Consequently, $$g(s)X(s)=g(s)\sum_{k>N(s)}\frac{a_k\widetilde\eta_{k,\rho}(s)}{k^{1/2+s}}+o(1),\quad s\to 0+\quad\text{a.s.}$$ Lemmas~\ref{4} and \ref{5} therefore imply $\limsup_{s\to0+}g(s)X(s)\leq 1$ a.s.\ after letting $\gamma\to 0+$.

For the converse inequality, take $\rho=1$ and use the decomposition
\begin{multline*}
X(s)=\sum_{k\leq N(s)}\frac{a_k\eta_k}{k^{1/2+s}}-\sum_{k\leq N(s)}\frac{a_k\widetilde\eta_{k,1}(s)}{k^{1/2+s}}\\+\sum_{k>N(s)}\frac{a_k}{k^{1/2+s}}\big(\eta_k\1_{\mathcal A_{k,1}(s)}
-\me[\eta_k\1_{\mathcal A_{k,1}(s)}]\big)+\sum_{k\geq1}\frac{a_k\widetilde\eta_{k,1}(s)}{k^{1/2+s}}.
\end{multline*}
Along the sequence $(\mathfrak s_n)$ from Lemma~\ref{lem:new}, the first three terms, after multiplication by $g(\mathfrak s_n)$, converge to zero almost surely by Lemmas~\ref{lem:prelim}, \ref{lem:l1}, and \ref{lem:inter1}. Hence, Lemma~\ref{lem:new} gives $$\limsup_{s\to0+}g(s)X(s)\geq 1-\delta\quad\text{a.s.}$$ Letting $\delta\to 0+$, we obtain $\limsup_{s\to0+}g(s)X(s)=1$ a.s. Applying the same argument to $-\eta_k$ yields $\liminf_{s\to0+}g(s)X(s)=-1$ a.s.

Finally, \eqref{eq:variance} implies $\log\log v(s)\sim\log\log(1/s)$ as $s\to0+$. Thus $g(s)$ may be replaced by $(2v(s)\log\log v(s))^{-1/2}$. The normalized random function is
almost surely continuous on a right neighbourhood of zero. Its $\limsup$ and $\liminf$ are $1$ and $-1$, respectively, so the intermediate value theorem shows that its cluster set is $[-1,1]$. This completes
the proof of Theorem \ref{thm:main4}.

\section{Justification for Remark \ref{rem:suff}}\label{sect:suff}

In this section, we show that Condition B entails Condition A, thereby justifying Remark \ref{rem:suff}.

Let $k_1<k_2<\ldots$ be the increasing enumeration of $\mathcal{S}$. According to part (a) of Condition B, for sufficiently large $j$, $$j=\rho(k_j)\leq C k_j/\varphi(k_j).$$ By part (b) of Condition B, $$Q_{k_j}=\frac{k_j(\log k_j)^\beta}{a_{k_j}^2 (\log\log k_j)^2\log^{(3)}k_j}\geq C_1\frac{k_j}{\varphi(k_j)}\geq C_2j.$$ Thus, $$\#\{k\in\mathcal{S}:~Q_k\leq x\}=\#\{j:~Q_{k_j}\leq x\}\leq \#\{j:~C_2j\leq x\}=O(x),\quad x\to\infty,$$ that is, Condition A does indeed hold.

\section{Applications of Theorem \ref{thm:main4}}\label{sect:applications}

We first provide examples of $(a_k)$ satisfying \eqref{eq:basic} and Condition B (hence, also Condition A). This means that Theorem \ref{thm:main4} applies to such sequences.

\noindent {\sc Logarithmic functions}: $a_k=(\log k)^\alpha$ for $k\geq 2$ and $\alpha>-1/2$, $a_1=0$. According to \eqref{eq:basic2}, the $(a_k)$ so defined satisfies \eqref{eq:basic} with $c=(2\alpha+1)^{-1}$ and $\beta=2\alpha+1$. In the notation of Condition B, we can take $\varphi(x)=1$. Then part (b) of Condition B trivially holds. Summarizing, Proposition \ref{prop:bura} follows from Theorem \ref{thm:main4}.

\noindent {\sc The von Mangoldt function} $\Lambda$ is defined by $\Lambda(n):=\log p$ if $n=p^j$ for some prime number $p$ and some $j\in\mn$, and $\Lambda(n):=0$ otherwise. Put $a_k=(\Lambda(k))^\alpha$ for $k\in\mn$ and $\alpha>0$. It is shown on p.~273 in \cite{Buraczewski etal:2023} that the $(a_k)$ so defined satisfies \eqref{eq:basic} with $c=(2\alpha)^{-1}$ and $\beta=2\alpha$. Let $\pi$ be the prime counting function. Using $\pi(x)\leq x$, we obtain $$\sum_{k=2}^{\lfloor \log_2 x\rfloor}\pi(x^{1/k})\leq x^{1/2}+\lfloor \log_2 x\rfloor x^{1/3}=O(x^{1/2}),\quad x\to\infty.$$ By the prime number theorem (see, for instance, Theorem 6.2.1 in \cite{BGT:1989}), $\pi(x)\sim x/\log x$ as $x\to\infty$. Hence, $$\rho(x)=\pi(x)+\sum_{k=2}^{\lfloor \log_2 x\rfloor}\pi(x^{1/k})=O(x/\log x),\quad x\to\infty,$$ and we may take $\varphi(x)=\log x$ in part (a) of Condition B. For $k=p^m\in\mathcal{S}$, $$a_k^2=(\log p)^{2\alpha}\leq(\log k)^{2\alpha}\leq C\frac{(\log k)^{2\alpha+1}}{(\log\log k)^2\log^{(3)}k}$$ for all sufficiently large $k$, that is, part (b) of Condition B holds.

Actually, our desire to understand whether a LIL holds for this particular $(a_k)$ has been the original motivation behind this work.

\noindent {\sc Some bounded $(a_k)$}. We state right away that every bounded sequence $(a_k)$ satisfies Condition B with, for example, $\varphi(x)=1$.   

\noindent (a) {\sc The M\"{o}bius function} $\mu$ is defined by $\mu(1)=1$, $\mu(n)=(-1)^k$ if $n$ is the product of $k$ distinct primes, and $\mu(n)=0$ if $n$ is divisible by the square of a prime. Put $a_k=\mu(k)$ for $k\in\mn$. Its support $\mathcal{S}$ is the set of square-free integers. By the classical asymptotic formula for the number of square-free integers (see, for instance, Theorem 334 in Chapter XVIII of \cite{Hardy+Wright:2008}), $\rho(x)=\sum_{n\leq x}(\mu(n))^2~\sim~6\pi^{-2}x$ as $x\to\infty$. Thus, we may take $\varphi(x)=1$. Moreover, partial summation in the last asymptotic formula yields $$\sum_{k=1}^n \frac{(\mu(k))^2}{k}~\sim~\frac{6}{\pi^2}\log n,\quad n\to\infty,$$ so \eqref{eq:basic} holds with $c=6\pi^{-2}$ and $\beta=1$. 

\noindent (b) {\sc The Liouville function} $\lambda$ is defined by $\lambda(1):=1$ and $\lambda(n):=(-1)^{\Omega(n)}$ for $n\geq2$, where, for the prime factorization $n=p_1^{\alpha_1}\cdots p_m^{\alpha_m}$, $\Omega(n):=\alpha_1+\cdots+\alpha_m$. In particular, $(\lambda(k))^2=1$ for each $k\in\mn$, and therefore
$$\sum_{k=1}^n \frac{\lambda(k)^2}{k}=\sum_{k=1}^n \frac{1}{k}~\sim~\log n,\quad n\to\infty.$$ Put $a_k=\lambda(k)$. Then \eqref{eq:basic} holds with $c=\beta=1$. We may take $\varphi(x)=1$. 

\noindent (c) {\sc Sums of two squares}. Define $$\mathcal{S}=\{k\in\mn:~k=m^2+n^2~\text{ for some }m,n\in\mathbb Z\}$$ and put $a_k=\1_{\{k\in\mathcal{S}\}}$ for $k\in\mn$. Landau's theorem (see, for instance, Theorem 7.28 on p.~261 in \cite{LeVeque:2002}) states that
\begin{equation}\label{eq:Landau-two-squares}
\rho(x)=\#\{k\in\mathcal S:~k\leq x\}~\sim~ K \frac{x}{(\log x)^{1/2}}, \quad x\to\infty,
\end{equation}
where
$$K:=\frac{1}{\sqrt{2}}
\prod_{\substack{p\ {\rm prime}\\p\equiv3\pmod4}}
\left(1-\frac{1}{p^2}\right)^{-1/2}
$$
is the Landau--Ramanujan constant. Thus, we may take $\varphi(x)=(\log x)^{1/2}$. By partial summation,
\begin{equation*}
\sum_{k=1}^n \frac{a_k^2}{k}=\frac{\rho(n)}{n}+\int_1^n\frac{\rho(t)}{t^2}\,{\rm d}t~\sim~ K\int_2^n\frac{{\rm d}t}{t(\log t)^{1/2}}~\sim~2K(\log n)^{1/2}, \quad n\to\infty.
\end{equation*}
Thus, \eqref{eq:basic} holds with $c=2K$ and $\beta=1/2$.

\noindent (d) {\sc Generalized divisor function $d_r$ with $r\in (0,1]$}. For $r>0$, let $d_r$ denote the generalized divisor function, see formulas~(1.1) and (1.2) in \cite{Norton1994}. It is the multiplicative
arithmetic function determined by
\begin{equation}\label{eq:dr-prime-powers}
d_r(1)=1,\quad d_r(p^\ell)=\prod_{j=1}^\ell \Big(1+\frac{r-1}{j}\Big), 
\quad \ell\in\mn
\end{equation}
for each prime $p$. In particular, $d_r(k)>0$ for every $k\in\mn$, and hence its support is $\mathcal{S}=\mn$. Thus, putting $a_k=d_r(k)$ for $k\in\mn$, we may take $\varphi(x)=1$.

We first verify \eqref{eq:basic} for $r>0$. The Dirichlet series of $d_r^2$ which is also multiplicative has the factorization
$$\sum_{k\geq1}\frac{(d_r(k))^2}{k^s}=\zeta(s)^{r^2}G_{r,2}(s),\quad {\rm Re}\, s>1,$$ where
\begin{equation}\label{eq:Grq}
G_{r,2}(s):=\prod_p (1-p^{-s})^{r^2}\sum_{\ell\geq 0}\frac{(d_r(p^\ell))^2}{p^{\ell s}}.
\end{equation}
Indeed, $$\sum_{\ell\geq 0}\frac{(d_r(p^\ell))^2}{p^{\ell s}}=1+\frac{r^2}{p^s}+O(p^{-2\,{\rm Re}\, s}),$$ so the Euler product in \eqref{eq:Grq} converges absolutely in every
half-plane ${\rm Re}\, s\geq 1/2+\varepsilon$. In particular,
$G_{r,2}(1)>0$. The Selberg-Delange theorem (see, for instance, Theorem~3 on p.~185 in \cite{Tenenbaum:1995}) therefore gives
\begin{equation}\label{eq:dr-moments}
\sum_{k=1}^n (d_r(k))^2~\sim~ \frac{G_{r,2}(1)}{\Gamma(r^2)}n(\log n)^{r^2-1}, \quad n\to\infty.
\end{equation}
Applying partial summation, we obtain
\begin{equation*}
\sum_{k=1}^n \frac{(d_r(k))^2}{k}=\frac{1}{n}\sum_{j=1}^n (d_r(j))^2+ \sum_{j=1}^{n-1}\sum_{k=1}^j (d_r(k))^2 \frac{1}{j(j+1)}~\sim~\frac{G_{r,2}(1)}{\Gamma(1+r^2)}(\log n)^{r^2},\quad n\to\infty.
\end{equation*}
Consequently, \eqref{eq:basic} holds with $c=G_{r,2}(1)/\Gamma(1+r^2)$ and $\beta=r^2$.

Assume now that $r\in (0,1]$. It follows from \eqref{eq:dr-prime-powers} that $d_r(p^\ell)\leq 1$ for $\ell\in\mn_0$. By multiplicativity, $0<d_r(k)\leq1$ for $k\in\mn$, that is, the sequence $(a_k)_{k\geq 1}$ is bounded.

\noindent (e) {\sc The Bessel function $J_\nu$ of the first kind} of order $\nu>-1$ is given by $$J_\nu(x):=\sum_{m\geq 0}\frac{(-1)^m}{m!\Gamma(m+\nu+1)}\Big(\frac{x}{2}\Big)^{2m+\nu},\quad x>0,$$ where $\Gamma$ is the Euler gamma function. Put $a_k=k^{1/2}J_\nu(k)$ for $k\in\mn$. By formula (1) on p.~199 in \cite{Watson:1944}, $$J_\nu(k)=\Big(\frac{2}{\pi k}\Big)^{1/2}\cos\Big(k-\frac{\pi \nu}{2}-\frac{\pi}{4}\Big)+O(k^{-3/2}),\quad k\to\infty,$$ whence $(J_\nu(k))^2=(\pi k)^{-1}(1+\sin(2k-\pi\nu))+O(k^{-2})$. In view of $$\sup_{n\geq 1}\Big|\sum_{k=1}^n \eee^{{\rm i}(2k-\pi\nu)}\Big|=\frac{1}{\sin 1},$$ an application of the Dirichlet test shows that the series $\sum_{k\geq 1} k^{-1}\sin(2k-\pi\nu)$ converges. Hence, $$\sum_{k=1}^n \frac{a_k^2}{k}=\sum_{k=1}^n (J_\nu(k))^2~\sim~ \frac{1}{\pi}\log n,\quad n\to\infty.$$ Thus, \eqref{eq:basic} holds with $c=\pi^{-1}$ and $\beta=1$. Finally, we may take $\varphi(x)=1$. 
Finally, we give an example of $(a_k)$ satisfying \eqref{eq:basic}, Condition A, but not Condition B.

\noindent {\sc Sequence with sparse spikes}. Put $$
a_k:=
\begin{cases}
\dfrac{2^{j/2}}{j}, & \text{for}~k=2^j,\quad j\in\mn,\\[1.2ex]
1, & \text{for other}~k.
\end{cases}
$$
Then
\begin{equation*}
\sum_{k=1}^n \frac{a_k^2}{k}=\sum_{k=1}^n \frac{1}{k}+\sum_{\substack{j\geq1\\2^j\leq n}}\Big(\frac{1}{j^2}-\frac{1}{2^j}\Big).
\end{equation*}
Since the series $\sum_{j\geq1}(j^{-2}-2^{-j})$ converges, we obtain $$\sum_{k=1}^n \frac{a_k^2}{k}~\sim~\log n,\quad n\to\infty.$$ Thus, \eqref{eq:basic} holds with $c=\beta=1$.

Recall the notation $$Q_k=\frac{k\log k}{a_k^2 (\log\log k)^2\log^{(3)}k},\quad k\in\mathcal{S}.$$ Suppose first that $k\in\mathcal{S}$ is not a power of $2$. Then $a_k=1$ and $$
Q_k=k\frac{\log k}{(\log\log k)^2\log^{(3)}k}.$$ Since the second factor diverges, we have $Q_k\geq k$ for all sufficiently large $k$ that are not powers of $2$. Consequently,
$$\#\{k\notin\{2^j:j\geq 4\}:~Q_k\leq x\}=O(x),\quad x\to\infty.$$

At an exceptional index $k=2^j$, we have $a_{2^j}^2=j^{-2}2^j$. Therefore,
\begin{equation*}
Q_{2^j}=\frac{2^j\log(2^j)}{(2^j/j^2)(\log\log(2^j))^2\log^{(3)}(2^j)}=\frac{j^3\log2}{(\log(j\log2))^2\log\log(j\log2)}.
\end{equation*}
In particular, $$\frac{Q_{2^j}}{j}=\frac{j^2\log2}{(\log(j\log2))^2\log\log(j\log2)}~\to~ \infty, \quad j\to\infty.$$ Hence, $Q_{2^j}\geq j$ for all sufficiently large $j$, and therefore $\#\{j\geq 4:~Q_{2^j}\leq x\}=O(x)$ as $x\to\infty$. Combining the estimates for the ordinary and exceptional indices, we conclude that $$\#\{k\geq 16:~Q_k\leq x\}=O(x),\quad x\to\infty,$$ that is, Condition A holds.

Since $\mathcal{S}=\{k\geq 16:~a_k\neq 0\}=\{16, 17,\ldots\}$, any positive function $\varphi$ satisfying part~(a) of Condition~B must be bounded. Hence, if Condition~B were satisfied, part~(b) would imply
$$a_k^2\leq C\frac{\log k}{(\log\log k)^2\log^{(3)}k}$$ for all sufficiently large $k$. Taking $k=2^j$ shows that the resulting
inequality $$\frac{2^j}{j^2}\leq C\log 2 \frac{j}{(\log(j\log2))^2\log\log(j\log2)}$$ fails for all sufficiently large $j$. Thus, Condition~B does not hold.

The latter example replaces our original example $a_k=d_r(k)$ for $r>1$, which motivated the introduction of Condition A. One can check that the $(a_k)$ so defined satisfies Condition A but does not satisfy Condition B. We refrain from reproducing the corresponding argument here.

\noindent \textbf{Acknowledgment.}
This work was supported by the National Research Foundation of Ukraine (project
2023.03/0059 ‘Contribution to modern theory of random series’).

\end{document}